\documentclass[11pt]{amsart}
\usepackage{polski}
\usepackage[T1]{fontenc}
\usepackage{amsmath, amsthm, amssymb, amsfonts, enumerate}
\usepackage{enumitem}
\usepackage[colorlinks=true,linkcolor=blue,urlcolor=blue]{hyperref}
\usepackage{dsfont}  
\usepackage{color}
\usepackage{geometry}
\usepackage{todonotes}
\usepackage{enumerate}

\newtheorem{theorem}{Theorem}[section]
\newtheorem{remark}[theorem]{Remark}
\newtheorem{assumption}[theorem]{Assumption}
\newtheorem{lemma}[theorem]{Lemma}
\newtheorem{proposition}[theorem]{Proposition}
\newtheorem{corollary}[theorem]{Corollary}
\newtheorem{definition}[theorem]{Definition}
\newtheorem{example}[theorem]{Example}
\theoremstyle{plain}

\newcommand{\field}[1]{\mathbb{#1}}

\newcommand{\R}{\field{R}}
\newcommand{\bH}{\field{H}}

\newcommand{\E}{\field{E}}
\renewcommand{\P}{\field{P}}

\usepackage[utf8]{inputenc}
\usepackage[english]{babel}
\usepackage{indentfirst}
\usepackage{graphicx}
\usepackage{pdfpages}
\usepackage[title]{appendix}
\usepackage{fancyhdr}

\usepackage{enumitem}
\usepackage{amsmath}
\usepackage{amssymb} 
\usepackage{bm}
\usepackage{mathrsfs}
\usepackage{amsthm}
\usepackage{comment}
\usepackage{hyperref}
\usepackage{dsfont} 
 
\usepackage{mathtools}
\DeclareMathOperator*{\argmin}{arg\,min}

\DeclareMathOperator*{\esssup}{ess\,sup}
\DeclareMathOperator*{\essinf}{ess\,inf}

\title[Strong solutions to submodular MFGs with common noise]{Strong solutions to submodular mean field games with common noise and related McKean-Vlasov FBSDEs} 
\author[Dianetti]{Jodi Dianetti}

\date{\today}

\numberwithin{equation}{section}
\begin{document}
\maketitle

\begin{abstract}
This paper studies multidimensional mean field games with common noise and the related system of McKean-Vlasov forward-backward stochastic differential equations deriving from the stochastic maximum principle.
We first propose some structural conditions which are related to the submodularity of the underlying mean field game and are a sort of opposite version of the well known Lasry-Lions monotonicity.   
By reformulating the representative player minimization problem via the stochastic maximum principle, the submodularity conditions allow to prove comparison principles for the forward-backward system, which correspond to the monotonicity of the best reply map.  
Building on this property, existence of strong solutions is shown via Tarski's fixed point theorem, both for the mean field game and for the related McKean-Vlasov forward-backward system.
In both cases, the set of solutions enjoys a lattice structure, with minimal and maximal solutions which can be constructed by iterating the best reply map or via the \emph{fictitious play} algorithm.        
\end{abstract}

\maketitle


\section{Introduction} 
Mean field games (MFGs, in short) have been proposed independently by \cite{HuangMalhameCaines06} and \cite{LasryLions07}, and arise as limit models for non-cooperative symmetric $N$-player games with interaction of mean field type as the number of players $N$ tends to infinity.
When the choice of the players is influenced by noises which are correlated, a common noise appears in the limiting MFG. 
While an extensive literature was developed around this topic (see, e.g., the two-volume book \cite{CarmonaDelarue18} and the references therein), a general understanding of the nature of the solutions to MFGs with common noise remains a central open problem in MFG theory. 
In particular, the existence of \emph{strong solutions} (i.e., solutions which are adapted to the common noise) is known mainly under conditions which imply their uniqueness (see \cite{ahuja2016wellposedness, ahuja.ren.yang.2019, CarmonaDelarueLacker16, DelarueFT19, huang.tang.2021, tchuendom}).

Crucially, despite mathematically very desirable, uniqueness in game theory represents an exceptional situation and multiple equilibria may arise. 
Hence, the main motivation of this paper is to (partially) overcome this limitation. 
This is done by proposing structural conditions and developing a theory which allows to study the existence and the approximation of strong solutions to MFGs with common noise, under lack of uniqueness. 

 
\subsection*{The results} 
Consider independent Brownian motions $W$ and $B$, which are independent from an initial condition $\xi$. 
For any $B$-adapted stochastic flow of probability measures $\mu$, a representative player minimizes, by choosing an admissible control $\alpha$, the cost functional
\begin{align}  
    \label{eq intro MFG problem}
& J(\alpha,\mu)= \mathbb{E} \bigg[ \int_0^T h(t,X_t, \mu_t, \alpha_t) dt + g(X_T, \mu_T) \bigg],  \\ \notag
& \text{subject to } dX_t=b(t,X_t, \mu_t, \alpha_t)dt + \sigma(t,X_t) dW_t + \sigma^{\circ} (t,X_t) dB_t,   \quad X_0=\xi. \notag
\end{align}
A strong MFG equilibrium (or solution) is a process $\mu$ which coincides with the flow of conditional probabilities $(\mathbb P _{\text{\tiny{$X_t^\mu|B$}}})_t$ with respect to the common noise $B$ of a trajectory $X^\mu$, with $X^\mu$ being optimal when optimizing against $\mu$ itself. 
By enforcing some additional conditions, one can reformulate the representative player minimization problem (parametrized by $\mu$) via the stochastic maximum principle (SMP, in short).   
Doing so, any optimal trajectory $X^\mu$ for $\mu$ corresponds to the forward component $X$ of a solution $(X,Y,Z,Z^\circ)$ of the fully coupled  forward-backward stochastic differential equation (FBSDE, in short)
\begin{align}
    \label{eq intro FBSDE optimal controls}
    & dX_t = b(t,X_t,\mu_t, \hat \alpha (t,X_t,\mu_t,Y_t))dt + \sigma (t,X_t)dW_t + \sigma^\circ (t,X_t)dB_t, \quad X_0=\xi, \\ \notag
    & dY_t = -D_x H(t,X_t,\mu_t, Y_t,\hat \alpha (t,X_t,\mu_t,Y_t)) dt + Z_t dW_t + Z_t^\circ dB_t, \quad  Y_T = D_x g(X_T, \mu_T), \notag 
\end{align} 
where $\hat \alpha (t,x,\nu,y) := \argmin_{a\in A} H(t,x,\nu,y,a) $ is the (unique, under suitable assumptions) minimizer of the Hamiltonian $H(t,x,\nu,y,a) := b(t,x,\nu,a)y + h(t,x,\nu,a)$.  
In this spirit, one wishes to relate any strong MFG equilibrium $\mu$ to a solution $(X,Y,Z,Z^\circ)$ 
to the conditional McKean-Vlasov (MKV, in short) FBSDE
\begin{align*}
   & dX_t = b(t,X_t,\mathbb P _{\text{\tiny{$X_t|B$}}}, \hat \alpha (t,X_t,\mathbb P _{\text{\tiny{$X_t|B$}}},Y_t))dt + \sigma (t,X_t)dW_t + \sigma^\circ (t,X_t)dB_t, \  X_0=\xi,  \\ \notag
    & dY_t = -D_x H(t,X_t,\mathbb P _{\text{\tiny{$X_t|B$}}}, Y_t,\hat \alpha (t,X_t,\mathbb P _{\text{\tiny{$X_t|B$}}},Y_t)) dt + Z_t dW_t + Z_t^\circ dB_t, \  Y_T = D_x g(X_T, \mathbb P _{\text{\tiny{$X_T|B$}}}). \notag 
\end{align*} 

This paper addresses the problems above by assuming some structural conditions which are related to the submodularity of the underlying mean field game. 
It is crucial to underline that (see \cite{dianetti.ferrari.fischer.nendel.2019} or Remark \ref{remark Lasry-Lions monotonicity condition} below) the submodularity property represents in MFGs a sort of antithetic version of the well known Lasry-Lions monotonicity condition, which is typically related to the uniqueness of the equilibria (see \cite{CarmonaDelarue18}). 
Indeed, the submodularity conditions appears (implicitly) in a number of linear-quadratic models (see, e.g., \cite{Bensoussan-etal-2016}) and in many works discussing multiple equilibria in MFGs (see, e.g., \cite{BardiFischer18, Cecchin&DaiPra&Fischer&Pelino19, DelarueFT19}), although this property is not exploited therein.  
 
One of the starting points of this work was to investigate whether the submodularity of the game could be related to a comparison principle, in the parameter $\mu$, of the related family of multidimensional FBSDEs \eqref{eq intro FBSDE optimal controls}.
When the state process is 1-dimensional, this idea can be easily motivated in the following example.   
Take $h(t,x,\nu,a) = f(t,x,\nu) + l(t,x,a)$ such that, for $\phi \in \{ f(t,\cdot,\cdot), g \}$, 
$\phi$ is  $C^1$ in $x$ and $\phi_x (x,\bar \nu) \leq   \phi_x (x, \nu)$, for $\nu \leq \bar \nu$ (in a suitable sense).
If $b$ and $\hat \alpha$ do not depend on the measure $\nu$, then it follows that
$$ 
 H_x (t,x,\bar \nu, y, \hat \alpha (t,x,y) ) \leq  H_x (t,x, \nu, y, \hat \alpha (t,x,y) ) \quad \text{and} \quad g_x ( x, \bar \nu) \leq g_x ( x, \nu), \quad  \text{for $\nu \leq \bar \nu$.}
$$ 
The latter inequality suggests to investigate a suitable comparison principle for the FBSDE \eqref{eq intro FBSDE optimal controls} which allows to conclude that $X^\mu \leq  X^{\bar \mu}$ when $\mu \leq \bar \mu$. 
Such a comparison principle is clearly related to the monotonicity of the \emph{best reply map} $\mu \mapsto (\mathbb P _{\text{\tiny{$X_t^\mu|B$}}})_t$. 

In the sequel, two classes of structural conditions are proposed, for either separable or nonseparable Hamiltonians.  
For separable Hamiltonians, the monotonicity of the best reply map is proved, from which a new comparison principle for  multidimensional FBSDEs follows.
When the Hamiltonians are nonseparable, we instead employ the comparison principle in \cite{chen.luo2021} in order to show that the best reply map is increasing.

Building on the monotonicity of the best reply map, the main results of this paper are the following: 
\begin{enumerate}
    \item Existence of strong solutions to the MFG problem and to the related MKV FBSDE is shown via Tarski's fixed point theorem (see \cite{T}). Moreover, both sets of solutions enjoy a lattice structure, and there exist minimal and maximal solutions. 
    \item The algorithm which consists in iterating the best reply map (equivalently, in solving iteratively the FBSDE \eqref{eq intro FBSDE optimal controls}) converges either to the minimal or to the maximal solution, when suitably initialized.
    \item The \emph{fictitious play algorithm} (which consists in solving a sequence of control problems, at each iteration solving the representative player minimization problem parametrized by an arithmetic average of the optimal distributions computed in the previous steps) converges either to the minimal or to the maximal solution, when suitably initialized. 
\end{enumerate}

The approach presented in this paper focuses on the reformulation of the representative player minimization problem via the SMP, and one might investigate similar results in terms of the --so called-- master equation or of the MFG system (see, e.g., \cite{CarmonaDelarue18}). 
These questions are left for future research. 
 
\subsection*{Related literature}  
We now discuss the relations between the results described above and the ones known in the literature.  

\subsubsection*{Strong solutions to MFGs with common noise} 
A probabilistic formulation of MFGs with common noise was presented in \cite{CarmonaDelarueLacker16}, together with several notions of equilibrium. 
There, the definition of a weak solution was introduced, where (in line with the theory of stochastic differential equations) the equilibrium measure flow is not necessarily adapted  to the filtration generated by the common noise, but only conditionally independent to it.  
In \cite{CarmonaDelarueLacker16}, the existence of weak solutions is established in a very general setting, while the existence of a strong solution is deduced (in a similar spirit of the Yamada-Watanabe theory) under the Lasry-Lions monotonicity condition, which automatically implies the uniqueness of the equilibrium.  

Along with \cite{CarmonaDelarueLacker16}, several methods and structural conditions have been proposed in order to study the existence of (strong) solutions.  
The MKV FBSDE deriving from the SMP was first used in \cite{carmona.delarue.2013} in order to solve linear-convex MFGs with no common noise.  
Later, strong solutions to MFGs with common noise have been studied via the MKV FBSDE in \cite{ahuja2016wellposedness}, for a model with terminal cost satisfying convexity and weak monotonicity properties. 
A similar approach was used in \cite{ahuja.ren.yang.2019, huang.tang.2021} for more general models, while linear-quadratic models were solved in \cite{DelarueFT19, tchuendom}.
Furthermore, one can study MFGs by using MKV FBSDEs in which the backward component represents the minimal cost of the optimization problem at equilibrium. 
This method was first used in \cite{carmona.lacker.2015} for MFGs with no common noise, and later employed in \cite{burzoni.campi.2021} in order to show the existence of a weak solution to MFGs with absorbing boundary conditions and common noise. 
MFGs with common noise are also studied via partial differential equations (PDEs, in short) (see, e.g., the recent \cite{cardaliaguet.souganidis.2022}) and through the  master equation (see, e.g., \cite{bertucci2021monotone, cardaliaguet.delarue.lasry.lions.2019, CarmonaDelarue18, gangbo.meszaros.mou.zhang.2022, mou.zhang.2022master.antimon}). 
In particular, the recent \cite{mou.zhang.2022master.antimon} provides a well posedness theory for the master equation related MFGs with common noise by introducing some \emph{anti-monotonicity conditions}. 
These conditions seem to be related to the sobmodularity conditions proposed in this paper, though none of the two sets of assumptions includes the other.  
Existence of strong solutions is also discussed in some examples in \cite{dianetti.ferrari.fischer.nendel.2019, dianetti.ferrari.fischer.nendel.2022unifying} for submodular MFGs.   

In relation to the previous works, this paper presents a theory establishing the existence of strong solutions without relying on uniqueness.  
Moreover, coefficients which may be discontinuous in the measure as well as degenerate settings are allowed:
In particular, first order MFGs with (see, e.g., \cite{cardaliaguet.souganidis.2022}) and without common noise can be treated. 

\subsubsection*{Approximation of strong solutions}
The problem of providing algorithms able to approximate MFG solutions was first addressed in \cite{CardaliaguetHadikhanloo17}, where the authors studied, through PDE methods, the convergence of the fictitious play in potential MFGs with no common noise.
This algorithm was further studied for MFGs of stopping and with absorbing boundary conditions in \cite{dumitrescu.leutscher.tankiv.2022linear} and, with the help of machine learning techniques, in \cite{elie2019approximate, perrin2020fictitious, xie2020provable}. 
Another learning procedure able to approximate MFG equilibria is the one studied, for submodular MFGs, in \cite{dianetti.ferrari.fischer.nendel.2019, dianetti.ferrari.fischer.nendel.2022unifying}. 
This algorithm consists in iterating the best reply map, and it seems to be quite promising when combined with reinforcement learning methods, as shown in the recent \cite{guo.hu.xu.zhamg.2019learning, lee.Rengarajan.Kalathil.Shakkottai2021}. 

In this work, both the learning procedure which consists in iterating the best reply map and the fictitious play are further developed for MFGs with common noise and for conditional MKV FBSDEs, and their convergence is studied by using stability results for BSDEs.
It is worth underlining that, in contrast with the papers mentioned above (but \cite{lee.Rengarajan.Kalathil.Shakkottai2021}), these convergence results hold under lack of uniqueness of the equilibria.
 
\subsubsection*{FBSDEs, MKV FBSDEs and comparison principles for FBSDEs} 
Solvability of coupled FBSDEs was first discussed in \cite{antonelli.93}
and soon has become a relevant topic in (applied) probability theory (see \cite{delarue.02,hu.peng.95,ma.protter.yong.94,pardoux.tang.99} and the monograph \cite{ma.yong.99}, among many others). 
An important use of FBSDEs in stochastic control theory comes from the SMP (see \cite{peng.90} or the monograph \cite{CarmonaDelarue18}).
This approach allows to characterize (under suitable linear-convex assumptions) the optimal control in terms of the minimizer of the Hamiltonian of the control problem and of the solution of the system.
Building on the SMP, MKV FBSDEs were introduced in \cite{carmona.dealarue.2013meanfieldFBSDEs, carmona.delarue.2013, carmona.delarue.2015forward} in  order to study control problems involving also the distribution of the controlled dynamics.
 
While comparison principles are classical for stochastic differential equations (see, e.g., \cite{Protter05}) and for backward stochastic differential equations (see, e.g., \cite{zhang2017} and the references therein), they are known to fail for FBSDEs (see, e.g., \cite{cvitanic.ma.96} or Example 6.2 at p.\ 22 in \cite{ma.yong.99}).  
When the backward component is one dimensional, two types of comparison principles are known.   
One is the comparison principle proved in \cite{cvitanic.ma.96,ma.wu.zhang.15}, which allows to compare the \emph{decoupling fields} of the systems (a function which allows to express the backward component in terms of the forward one).
Such results typically hinge on comparison principles for a suitable parabolic partial differential equation, which is solved by the decoupling field.
The second  type of comparison principles allows to compare both the forward and the backward component (see \cite{wu.xu.09}).
In multidimensional cases, comparison principles are even more difficult to obtain, and strong monotonicity assumptions are required.
To the best of the author's knowledge, such results were obtained, in multidimensional settings, only in \cite{chen.luo2021,wu.xu.09} for the forward and the backward components. 
 
In this paper, a comparison principle is obtained for multidimensional parameter-dependent-FBSDEs corresponding (via the SMP) to minimization problems which are submodular in the parameter (for further details, see Corollary \ref{corollary comparison principle} and Remark \ref{remark comparison principle} below). 
Moreover, existence of minimal and maximal strong solutions to the conditional MKV FBSDE is established by lattice theoretical methods.
This approach does not require the time horizon to be small and the Lipschitz continuity in the measure flow (see \cite{CarmonaDelarue18}), or to enforce the --so called-- monotonicity conditions (see \cite{ahuja2016wellposedness,ahuja.ren.yang.2019, huang.tang.2021}).  
  
\subsubsection*{Submodular games and MFGs}  
In game theory, the submodularity property was first introduced by Topkis \cite{To} for static $N$-player games and describes games in which players minimize cost functions which are submodular in the vector of strategies chosen by all the players. 
Such a class of games plays an important role in economical applications (see the monograph \cite{Vives01} and the references in it).  
In MFGs, the submodularity has already been used in \cite{Adlakhaetal} for a class of stationary discrete time games, in \cite{Wiecek17} for a class of finite state MFGs with exit, in \cite{CarmonaDelarueLacker16} for optimal timing MFGs and in \cite{dianetti.ferrari.fischer.nendel.2019} for MFGs involving  one-dimensional It\^o-diffusions.     
As shown in the recent \cite{dianetti.ferrari.fischer.nendel.2022unifying}, the submodularity property allows to treat qualitatively different formulations of MFGs (such as MFGs with singular controls, optimal stopping, reflecting boundary conditions and finite-state problems).  
  
In terms of methodology, the closest contribution to the present work is \cite{dianetti.ferrari.fischer.nendel.2019} and a detailed comparison is, at this point, necessary.  
First of all, this paper provides a theory for MFGs with common noise and interprets the submodularity property from the FBSDE point of view, while \cite{dianetti.ferrari.fischer.nendel.2019} does not. 
In particular, extra technical care is needed in order to deal with stochastic flows of measures and the approximation result is proved using stability of backward stochastic differential equations.  
Furthermore, this paper also establishes the convergence of the fictitious play, which was not studied in \cite{dianetti.ferrari.fischer.nendel.2019}. 
Secondly, 
while the state space in \cite{dianetti.ferrari.fischer.nendel.2019} must be one-dimensional, here multidimensional settings are considered. 
This important issue is overcome by formulating the fixed point problem on a lattice of state processes (see Remark \ref{remark multidimensional} below).
Moreover, in the present paper, structural conditions are presented for which the drift of the state process can depend on the measure and in which the Hamiltonian is nonseparable, while the setting in \cite{dianetti.ferrari.fischer.nendel.2019} excludes this possibility (apart very specific examples). 
Finally, differently to \cite{dianetti.ferrari.fischer.nendel.2019}, this paper also allows to treat problems where the volatility of the state process is state dependent.

\subsection*{Outline of the paper} The rest of the paper is organized as follows. 
Section \ref{section Problem formulation and main results} presents the FBSDE-approach to submodular MFGs, by introducing the submodularity conditions and by stating both some intermediate results and the main results. 
The intermediate results are proved in Section \ref{section proof auxiliary results}, while the main results are proved in Section \ref{section proof main results}.
Examples are discussed in Section \ref{section examples}, while an auxiliary lemma is proved in Appendix \ref{appendix}.

\subsection*{General notation} 
Set $\mathbb N _0 :=\mathbb N \setminus \{ 0 \}$. 
For $d \in \mathbb{N}_0$ and $x,y \in \R^d$, denote by  $xy$  the scalar product in $\R^d$, as well as by $|\cdot|$ the Euclidean norm in $\R^d$.  
For $x \in \mathbb{R}^d$,   $x^{\text{\tiny {$\top$}}}$ indicates the transpose of $x$. 
For $x,y \in \R^d$ and $c \in \R$,  set  $x \leq y$ if $x^i \leq y^i$ for each $i=1,...,d$, as well as $x \leq c$ if $x^i \leq c $ for each $i=1,...,d$. 
Moreover, define $x \land y := (x^1\land y^1,...,x^d\land y^d)$ and $x \lor y := (x^1 \lor y^1,...,x^d \lor y^d)$, where  $x^i \land y^i :=\min\{ x^i , y^i \}$ and $x^i \lor y^i :=\max\{ x^i , y^i \}$ for each $i=1,...,d$.
For a separable metric space $(E, \ell)$, $\mathcal P (E)$ denotes the set of probability measures on the Borel sets of $E$.
For $p >0 $,   $\mathcal P_p (E ) $ is the set of $\nu \in \mathcal P (E)$ such that $\int _{E} \ell^p(y,y_0) d \nu (y) < \infty$ for some (and thus for all)  $y_0 \in E$. 
The $p$-Wassertein distance on $\mathcal P _p  (E)$ is defined as
\begin{equation}
    \label{eq Wasserstain distance}
\ell_{E,p} (\bar \nu, \nu ) := \inf \bigg\{ \bigg( \int_{E \times E} \ell^p(x,y) d N (x,y) \bigg) ^{1/p} \, \Big| \,  N \in \Pi(\bar \nu, \nu) \bigg\},
\quad \text{for $\bar \nu, \nu \in \mathcal P _p (E)$,}
\end{equation} 
where  $\Pi(\bar \nu, \nu)$ denotes the set of probability measures on $E \times E$ with marginals $\bar \nu$ and  $\nu$. 
For $d \in \mathbb{N}_0$, $p>0$ and $\mu \in \mathcal P (\R^d)$, set $\| \mu \|_p := \int_{\R^d} |y|^p d\mu(y)$. 
Unless otherwise stated, $C$ indicates a generic positive constant, which may change from line to line.

\section{Problem formulation and main results}
\label{section Problem formulation and main results}
\subsection{MFG problem and related MKV FBSDE} 
Take $d, d_1, d_2, k \in \mathbb N _0$, a $d_1$-dimensional Brownian motion $W=(W^{1}, ...,W^{d_1})$ and a $d_2$-dimensional Brownian motion $B=(B^{1}, ...,B^{d_2})$ on a complete probability space $(\Omega,\mathcal F,\mathbb P)$. 
Assume the processes $W$ and $B$ to be independent, and independent from a $d$-dimensional square integrable random variable $\xi=(\xi^1,...,\xi^d)$.
Denote by $\mathbb F := (\mathcal F _t) _{t \in [0,T]}$ (resp.\ $\mathbb F^{\text{\tiny{$B$}}}:= (\mathcal F _t^{\text{\tiny{$B$}}}) _{t \in [0,T]}$) the right continuous extension of the filtration generated by $(W,B,\xi)$ (resp.\ $(B,\xi)$), augmented by $\mathbb P $-null sets.

Consider a set $A \subset \R^k$, $A=A^1\times...\times A^k$ for closed intervals $A^i\subset \mathbb{R}$,  $i=1,...,k$. 
The set of admissible (open loop) controls $\mathcal A$ is defined as
$$
\mathcal A := \{ \alpha : \Omega \times [0,T] \to A | \text{ $\alpha=(\alpha^1, ..., \alpha ^k) $ is square integrable $\mathbb F$-progressively measurable} \}.
$$
On the set $\mathcal P _2 (\R^d)$, consider the $2$-Wasserstein distance  $\ell_{\R^d,2}$. 
The coefficients of the problem are given through (deterministic) measurable functions 
\begin{align*}
(b, h) :&\, [0,T] \times  \mathbb{R}^d \times \mathcal{P}_2(\R ^d) \times A \rightarrow \mathbb{R}^d \times \R, \quad b:=(b^1,...,b^d), \\
(\sigma, \sigma^\circ) :&\, [0,T] \times  \mathbb{R}^d \rightarrow \mathbb{R}^{d \times d_1} \times \mathbb{R}^{d \times d_2},  \quad \sigma:=(\sigma _j ^i)_{\text{\tiny{$j=1,...,d_1$}}}^{\text{\tiny{$i=1,...,d$}}},\ \sigma^\circ:=(\sigma_j^{i,\circ})_{\text{\tiny{$j=1,...,d_2$}}}^{\text{\tiny{$i=1,...,d$}}},\\  
g :&\,  \mathbb{R}^d \times \mathcal{P}_2(\R ^d)\rightarrow \R,
\end{align*}
which will be subject to the assumptions below (see Assumptions \ref{assumption existence}, \ref{assumption uniqueness FBSDEs}, \ref{assumption for J}, \ref{assumption for comparison} and \ref{assumption a priori estimates}).

Introduce the finite measure $\pi:= \delta_0+dt+\delta_T$ on the Borel $\sigma$-algebra  of the interval $[0,T]$, where $\delta_0$ and $\delta_T$ denote the Dirac measures at time $0$ and $T$, respectively.   
The set of possible (stochastic) distributions of players is denoted by
\begin{align*}  
\mathcal M _{\text{\tiny{$B$}}}^2 := \Big\{ \mu: \Omega \times [0,T] \to & \mathcal P _2 (\R^d) \big|  \text{ $\mu$ is $\mathbb F ^{\text{\tiny{$B$}}}$-progr.\ meas.\ and }\E \big[ \begin{matrix} \int_0^T \| \mu_t \|_2^2 d\pi(t) \end{matrix} \big] < \infty   \Big\},
\end{align*}
while any controlled state processes will be an element of the set
$$
\mathbb H ^2 := \Big\{ M: \Omega \times [0,T] \to \R^d \big|\, \text{$M$ is  $\mathbb F $-progr.\ meas.\ and } \E \big[ \begin{matrix} \int_0^T | M_t |^2 d\pi(t) \end{matrix} \big] < \infty  \Big\}.
$$
The sets  $\mathcal M _{\text{\tiny{$B$}}}^2$ and $\mathbb H ^2$ are actually sets of equivalence classes of processes; that is, elements of $\mathcal M _{\text{\tiny{$B$}}}^2$ (and of $\mathbb H ^2$) are identified if they coincide $\mathbb P \otimes \pi$-a.e.\ in $\Omega \times [0,T]$. 
Also, for technical reasons, the processes in  $\mathcal M _{\text{\tiny{$B$}}}^2$ and in $\mathbb H ^2$ are not assumed to be continuous. 
For an $\mathbb F$-progressively measurable process $M$ (taking values in any Polish space), similarly to Remark 1 in \cite{tchuendom} one can show that the conditional law  $\mathbb P [ M_t \in \cdot \, |\mathcal F _t^{\text{\tiny{$B$}}}]$ of $M_t$ given the filtration $\mathcal F _t^{\text{\tiny{$B$}}}$ is such that $\mathbb P [ M_t \in \cdot \, |\mathcal F _t^{\text{\tiny{$B$}}} ] = \mathbb P [ M_t \in \cdot \, |\mathcal F _T^{\text{\tiny{$B$}}}] $ $\mathbb P$-a.s., for any $t \in[0,T]$.    
Therefore, in the sequel, $\mathbb P [ M_t \in \cdot \, |\mathcal F _t^{\text{\tiny{$B$}}}]$ will be simply denoted by $\mathbb P _{\text{\tiny{$M_t|B$}}}$.

\subsubsection{The MFG problem} 
Consider a mean field game (MFG, in short) in which, for any $\mu \in \mathcal M _{\text{\tiny{$B$}}}^2$, the representative player minimizes the cost functional
\begin{equation*}  
 J(\alpha,\mu):= \mathbb{E} \bigg[ \int_0^T h(t,X_t, \mu_t, \alpha_t) dt + g(X_T, \mu_T) \bigg], \quad \alpha \in \mathcal A,
\end{equation*}
where the process $X=(X^{1},...,X^d)$ denotes the solution to the controlled stochastic differential equation (SDE$_{(\alpha,\mu)}$, in short)
\begin{equation}\label{eq controlled SDE} 
dX_t=b(t,X_t, \mu_t, \alpha_t)dt + \sigma(t,X_t) dW_t + \sigma^{\circ} (t,X_t) dB_t,  \quad t \in [0,T], \quad X_0=\xi.
\end{equation}
For the moment, just assume that, for any $\alpha \in \mathcal A$ and $\mu \in \mathcal M _{\text{\tiny{$B$}}}^2$, the SDE$_{(\alpha, \mu)}$ admits a unique strong solution and the integrals in the cost functional $J(\alpha, \mu)$ are well defined, possibly equal to $+\infty$.
For $\mu \in \mathcal M _{\text{\tiny{$B$}}}^2$, a control $\alpha ^\mu \in \mathcal A$ is said to be optimal for $\mu$ if $J(\alpha^\mu, \mu) \leq J(\alpha, \mu)$ for any other $\alpha \in \mathcal A$. 
When such a control exists, one refers to the solution $X^\mu$ to the SDE$_{(\alpha^\mu, \mu)}$ \eqref{eq controlled SDE} as to optimal trajectory, and to the couple $(X^\mu,\alpha^\mu)$ as to optimal pair.

The focus of this paper is on the following notion of equilibrium.
\begin{definition}\label{def equilibrium}
A (strong) MFG equilibrium is a process $\mu\in \mathcal M _{\text{\tiny{$B$}}}^2$ such that 
$$
\mu_t = \mathbb P _{\text{\tiny{$X_t^\mu|B$}}}, \quad \text{ for any $t\in [0,T]$,} 
$$
for some optimal trajectory $X^\mu$.
\end{definition}
Since $\mathbb P _{\text{\tiny{$X_t^\mu|B$}}} =\mathbb P [ X_t^\mu \in \cdot \, |\mathcal F _t^{\text{\tiny{$B$}}}]$, any equilibrium $\mu$ as in Definition \ref{def equilibrium} is actually strong (see \cite{CarmonaDelarueLacker16}), in the sense that it is adapted to the filtration generated by the common noise $B$.

\subsubsection{The related FBSDEs}
We next introduce the forward-backward stochastic differential equation (FBSDE, in short) related to the MFG problem with common noise.
Define the Hamiltonians of the minimization problems by
$$
H(t,x,\mu,y,a) = b(t,x,\mu,a)y + h(t,x,\mu,a), \quad (t,x,\mu,y,a) \in [0,T] \times \R^d \times \mathcal P _2 (\R^d) \times \R^d \times A,
$$
and assume for the moment that the function $a \mapsto H(t,x,\mu,y,a)$ admits a unique minimizer $\hat \alpha (t,x,\mu,y)$ (see Assumption \ref{assumption existence} and Remark \ref{remark Lipschitz feedback} below). 
Namely, set
\begin{equation*}\label{eq optimal feedback} 
\hat \alpha (t,x,\mu,y) := \argmin_{a\in A} H(t,x,\mu,y,a) \in A, \quad (t,x,\mu,y) \in [0,T] \times \R^d \times \mathcal P _2 (\R^d) \times \R^d.
\end{equation*}  
For any $\mu \in  \mathcal M _{\text{\tiny{$B$}}}^2$, following the stochastic maximum principle approach, one can characterize the optimal pair for $\mu$ in terms of the solution  $(X,Y,Z, Z^\circ)$ to the  FBSDE$_\mu$
\begin{align}\label{eq FBSDE optimal controls}
   & dX_t = b(t,X_t,\mu_t, \hat \alpha (t,X_t,\mu_t,Y_t))dt + \sigma (t,X_t)dW_t + \sigma^\circ (t,X_t)dB_t,  \\ \notag
    & dY_t = -D_x H(t,X_t,\mu_t, Y_t,\hat \alpha (t,X_t,\mu_t,Y_t)) dt + Z_t dW_t + Z_t^\circ dB_t,  \notag 
\end{align}
with boundary conditions $ X_0=\xi$ and $ Y_T = D_x g(X_T, \mu_T)$.
Recall that such a solution is defined as a $\mathbb F$-progressively measurable process $(X,Y,Z, Z^\circ): \Omega \times [0,T] \to \R^d \times \R^d \times \R^{d \times d_1} \times \R^{d \times d_2}$ 
such that
\begin{equation}\label{eq square integrability sol FBSDE}
 \E \bigg[ \sup_{t\in [0,T]}( | X _t |^2 + |  Y _t|^2 ) 
    + \int_0^T ( |  Z _t |^2+ |  Z _t^\circ |^2 )dt \bigg] < \infty,
\end{equation} 
and such that system \eqref{eq FBSDE optimal controls} and the boundary conditions hold.  
Also, for future reference, it is convenient to introduce the notation
\begin{align}
    \label{eq definition data FBSDE}
    \hat b (t,x,\mu,y) & := b(t,x,\mu, \hat \alpha (t,x,\mu,y)), \\ \notag
    \hat h (t,x,\mu,y) &:= D_x H (t,x,\mu,y, \hat \alpha (t,x,\mu,y)), \\ \notag
    \hat g (x,\mu) &: =  D_x g(x, \mu), \notag
\end{align}
which allows to rewrite the system \eqref{eq FBSDE optimal controls} in the more compact form
\begin{align*}
   & dX_t = \hat b(t,X_t, \mu_t ,Y_t)dt + \sigma (t,X_t)dW_t + \sigma^\circ (t,X_t)dB_t,  \quad X_0=\xi, \\ \notag
    & dY_t = - \hat h (t,X_t,\mu_t , Y_t) dt + Z_t dW_t + Z_t^\circ dB_t, \quad  Y_T = \hat g(X_T,\mu_T ).  
\end{align*} 
In this spirit,  
one wishes to relate any (strong) equilibrium $\mu$ of the MFG problem with a (strong) solution $(X,Y,Z, Z^\circ)$ of the following the McKean-Vlasov FBSDE (MKV FBSDE, in short) 
\begin{align}\label{eq MKV FBSDE}
   & dX_t = \hat b(t,X_t, \mathbb P _{\text{\tiny{$X_t|B$}}} ,Y_t)dt + \sigma (t,X_t)dW_t + \sigma^\circ (t,X_t)dB_t,  \quad X_0=\xi, \\ \notag
    & dY_t = - \hat h (t,X_t,\mathbb P _{\text{\tiny{$X_t|B$}}} , Y_t) dt + Z_t dW_t + Z_t^\circ dB_t, \quad  Y_T = \hat g(X_T,\mathbb P _{\text{\tiny{$X_T|B$}}} ).  
\end{align}
A solution of the MKV FBSDE is a solution $(X,Y,Z, Z^\circ)$ of the FBSDE$_\mu$ \eqref{eq FBSDE optimal controls} with $\mu = (\mathbb P _{\text{\tiny{$X_t|B$}}})_{t \in [0,T]}$. 

\subsection{Optimal controls, FBSDEs and best reply maps}  
In order to reformulate the representative player control problem in terms of the the FBSDE$_\mu$, it is natural to work under some regularity conditions.  
\begin{assumption}\label{assumption existence} 
For constants $K, \kappa, \lambda>0$ and fixed $(\nu_0, a_0) \in \mathcal P _2 (\R^d) \times A$, assume that:
\begin{enumerate} 
\item $b(t,x,\mu,a)=b_1(t,x,\mu) + b_2 (t) a$ with $b_1$ continuous, $h(t,x,\mu,\cdot)$ is convex in $a$, $h$ is lower semi-continuous in $(x,\mu,a)$ and $g$ is lower semi-continuous in $(x,\mu)$.
\item\label{assumption existence: Lipschitzianity data SDE} The functions $b_1, \sigma, \sigma ^\circ $ are $C^1$ w.r.t.\ $x$, and
\begin{align*} 
    & |b_2(t)| + |b_1(t,0,\nu_0 )| +| \sigma(t,0)|  +|  \sigma^\circ (t,0)|  \leq K ,\\
    & |D_x b_1 (t,x,\mu)| +| D_x \sigma(t,x)|  +| D_x \sigma^\circ (t,x)|  \leq K, 
\end{align*}   
for any $(t,x,\mu) \in [0,T] \times \R ^d \times \mathcal P _2(\R^d)$.
\item\label{assumption existence: growth conditions on costs} The function $h(t,\cdot,\mu,\cdot)$ is $C^1$ in $(x,a)$, the function $g(\cdot, \mu)$ is $C^1$ in $x$ and  
 \begin{align*}
  & \kappa (|a|^2 -1) \leq h(t,x,\mu,a) \leq K (1+|x|^2 + \| \mu \|_2^2 + |a|^2 ), \\
   & -\kappa \leq g (x,\mu) \leq K (1+|x|^2 + \| \mu \|_2^2 + |a|^2 ), \\
    & |D_x h(t,x,\mu,a)| + |D_a h(t,x,\mu,a)| + |D_x g(x,\mu)| \leq K (1+|x| + \| \mu \|_1 + |a| ), 
\end{align*}   
for any $(t,x,\mu, a) \in [0,T] \times \R ^d \times \mathcal P _2(\R^d) \times A$.  
 Moreover, the functions $D_x h (t,\cdot, \mu, \cdot)$, $D_a h (t,\cdot, \mu, \cdot)$ and $ D_x g(\cdot, \mu)$ are Lipschitz continuous in $(x,a)$, uniformly in $(t,\mu)$. 
\item  The function  $h$ is $\lambda$-convex in $a$; i.e.,
\begin{align*}
& h(t,x,\mu,\bar a) - h(t,x,\mu, a) -  D_a h(t,x,\mu, a) (\bar a -a ) \geq \lambda |\bar a -a |^2,
\end{align*} 
for any  $\bar a ,  a \in A$ and $(t,x,\mu) \in [0,T] \times \R^d \times \mathcal P _2(\R^d)$.
\end{enumerate} 
\end{assumption}
In particular, these conditions ensure some regularity of the feedback $\hat \alpha$.
\begin{remark}\label{remark Lipschitz feedback}
According to Lemma 3.3 at p.\ 137 Vol.\ I in  \cite{CarmonaDelarue18} , $\hat{\alpha}(t,x,\mu,y)$ is a singleton, and the function $(t,x,\mu,y) \mapsto \hat \alpha (t,x,\mu,y)$ is measurable, locally bounded and such that
\begin{align*}
&|\hat \alpha (t,\bar x,\mu, \bar y) - \hat \alpha (t,x,\mu,y)| \leq L (|\bar x -x | + |\bar y -y| ) ,  \\
& |\hat \alpha (t,x,\mu,y)| \leq L( 1 + |x| + \|\mu\|_1 + |y|), 
\end{align*}
for a constant $L>0$, for any $\bar x, x, \bar y, y \in \R^d$ and $(t,\mu) \in [0,T] \times \mathcal P _2(\R^d)$.
\end{remark}
For $\mu \in \mathcal M _{\text{\tiny{$B$}}}^2$, since the process $\mu$ is adapted the noise $B$, system FBSDE$_\mu$ \eqref{eq FBSDE optimal controls} can be seen as an FBSDE with deterministic coefficients. 
Then, referring to Definition 1.8 at p.\ 16 in Vol.\ II of \cite{CarmonaDelarue18}, we enforce the following requirement,
which is satisfied in linear-convex models or under some nondegeneracy of the diffusion terms (see the discussion in Subsection \ref{subsection Sufficient conditions for strong uniqueness for FBSDEs}).
\begin{assumption}
    \label{assumption uniqueness FBSDEs}
    For any $\mu \in \mathcal M _{\text{\tiny{$B$}}}^2$, the strong uniqueness holds for the FBSDE$_\mu$ \eqref{eq FBSDE optimal controls}.
\end{assumption} 

The previous condition, together with Remark \ref{remark Lipschitz feedback}, ensures the existence of optimal controls in strong formulation; that is, controls which are adapted to the noises and to the initial condition, and that can be realized on the underlying (fixed) filtered probability space.  
This result is stated in the following lemma. 
Its proof, which exploits the --so called-- \emph{compactification method} and the stochastic maximum principle, is difficult to refer and it is provided in Appendix \ref{appendix}.  
\begin{lemma}[Existence and characterization of optimal controls]\label{lemma existence of optimal controls}
Under Assumptions \ref{assumption existence}, \ref{assumption uniqueness FBSDEs}, for any $\mu \in \mathcal M _{\text{\tiny{$B$}}}^2$, there exists a unique optimal control 
$\alpha^\mu$. 
Moreover, such a control is determined by $\alpha^\mu = (\hat \alpha (t,X_t,\mu_t,Y_t))_{t\in [0,T]}$, where  $(X,Y,Z, Z^\circ)$ is the unique solution to the FBSDE$_\mu$ \eqref{eq FBSDE optimal controls}.
\end{lemma}

Thanks to the previous lemma, one can define the \emph{best reply maps} $\Gamma, \, \mathcal R$ and $R$ as    
\begin{align}\label{eq definition best reply maps}
    \Gamma ( \mu) & := X^\mu,  \quad \mathcal R (\mu) :=  (\mathbb P _{\text{\tiny{$X_t^\mu|B$}}})_{ t\in [0,T]}, \quad \mu \in  \mathcal M _{\text{\tiny{$B$}}}^2, \\ \notag 
    R(M) & := \Gamma ( (\mathbb P _{\text{\tiny{$M_t|B$}}})_{ t\in [0,T]}), \quad M \in  \mathbb H ^2,
\end{align} 
where $X^\mu$ denotes the optimal trajectory for $\mu$.
In this framework, MFG equilibria coincide with fixed points of the best reply map $\mathcal R$; namely,
\begin{equation*} 
    \label{eq MFG as fixed points of best response map measure process}
    \text{ $\mu \in  \mathcal M _{\text{\tiny{$B$}}}^2$ is a MFG equilibrium if and only if $\mu = \mathcal R (\mu)$.}
\end{equation*} 
\begin{remark}[Characterization of MFG equilibria]\label{remark characterization of MFGE}
Define the projection map 
$p:\bH^2 \to  \mathcal M _{\text{\tiny{$B$}}}^2,$ 
$p(M):=(\mathbb P _{\text{\tiny{$M_t|B$}}})_{ t\in [0,T]}$. 
and observe that, in this notation, we have
$$
R(M)= \Gamma \circ p (M), \text{ for any } M \in \bH^2, \quad \text{and} \quad \mathcal R (\mu) = p \circ \Gamma (\mu), \text{ for any } \mu \in \mathcal M _{\text{\tiny{$B$}}}^2.
$$
It follows that, if $\mu \in \mathcal M _{\text{\tiny{$B$}}}^2$ is a fixed point of $\mathcal R$, then $\Gamma(\mu)$ is a fixed point of $R$. Analogously, if $M \in \bH^2$ is a fixed point of  $R$, then the flow $p(M)$ is a fixed point of the map $\mathcal R$.
In other words, the function $p$ is a bijection (with inverse $\Gamma$) between the sets 
\begin{equation}
  \{ \text{Fixed points of $R$} \} \xrightarrow{p} \{ \text{Fixed points of $\mathcal R$} \} = \{ \text{MFG equilibria} \}.
\end{equation}
Moreover, Lemma \ref{lemma existence of optimal controls} has a relevant immediate consequence: The process $\mu \in  \mathcal M _{\text{\tiny{$B$}}}^2$ is a MFG equilibrium if and only if $\mu  =(\mathbb P _{\text{\tiny{$X_t|B$}}}) _{t \in [0,T]}$ with $(X,Y,Z, Z^\circ)$ solution to the MKV FBSDE \eqref{eq MKV FBSDE}.
Also, the process $(X,Y,Z, Z^\circ)$ is a solution to the MKV FBSDE \eqref{eq MKV FBSDE} if and only if $R(X)=X$.
\end{remark}

\subsection{Submodularity and monotonicity of the best reply maps} 
On the set $ \mathcal P _2 (\R^d)$, introduce the \emph{first order stochastic dominance} $\leq^{\text{\tiny{\rm st}}}$, which is defined, for $\mu, \, \bar \mu \in \mathcal P _2 (\R^d)$, as   
$$
\text{ $\mu \leq^{\text{\tiny{\rm st}}} \bar \mu \ $ if and only if $ \ \int_{\R^d} \varphi (x) d\mu(x) \leq  \int_{\R^d} \varphi (x) d\bar \mu (x)$ for any $\varphi \in \Phi$.} 
$$
Here $\Phi$ is the set of measurable functions $\varphi: \R^d \to \R$ which are nondecreasing (i.e., $\varphi(x) \leq \varphi( \bar x)$ if $x \leq \bar x$).  
The relation $\leq^{\text{\tiny{\rm st}}}$ is actually an order relation on the set $\mathcal P _2 (\R^d)$, and $(\mathcal P _2 (\R^d), \leq^{\text{\tiny{\rm st}}})$ is a partially ordered set.     
We also observe that, for two square integrable random variables $\zeta, \bar \zeta$, we have
\begin{equation}
    \label{eq monotonicity of conditional distribution}
    \text{ $\mathbb P _{\text{\tiny{$\zeta |B$}}} \leq^{\text{\tiny{\rm st}}} \mathbb P _{\text{\tiny{$\bar \zeta |B$}}} \quad \mathbb P$-a.s.$\quad $if$\quad \zeta \leq \bar \zeta \quad \mathbb P$-a.s.}
\end{equation}

The following two classes of assumptions are the key structural conditions which will allow to employ a lattice theoretical approach. 

\begin{assumption}[Separable Hamiltonian and mean field independent dynamics]\label{assumption for J} \
\begin{enumerate} 
    \item \label{assumption structural for J} (Separability) $k=d$ and, for $i=1,...,d$, one has 
    \begin{enumerate}
\item $b^i(t,x,\mu,a) = b_1^i(t,x^i) + b_2^i (t) a^i$ for functions 
$
b_1^i: [0,T] \times \mathbb{R}  \rightarrow  \R$, $ b_2^i:\, [0,T]  \rightarrow  \R  $;
\item $\sigma^{i}= (\sigma^i_j)_{j=1,...,d_1}$ and $\sigma^{i,\circ}= (\sigma^{i,\circ}_j)_{j=1,...,d_2}$ only depend on $x^i$;
\item $ h(t,x,\mu,a) = f(t,x,\mu) + \sum_{i=1}^d l^i(t,x^i,a^i)$ for functions $l^i: \, [0,T] \times \mathbb{R} \times A^i \rightarrow \mathbb{R}$, $f:\, [0,T] \times \mathbb{R}^d \times \mathcal{P}_2(\mathbb{R}^d)\rightarrow \mathbb{R}$.
    \end{enumerate}
    \item \label{assumption submodularity multidimensional for J} (Submodularity conditions)
For $dt$ a.a.\ $t \in [0,T]$, for $\phi \in \{f(t,\cdot, \cdot), g \}$, we have
\begin{enumerate}
    \item \label{assumption submodularity multidimensional for J: submodular in x} $\phi({x} \land \bar x,\mu )  + \phi({x} \lor \bar x,{\mu})  \leq \phi(x,\mu) + \phi(\bar{x},{\mu})  $ for all $ {x}, \bar x \in \mathbb{R}^d$ and $ \mu \in \mathcal{P}_2(\mathbb{R}^d)$;  
    \item \label{assumption submodularity multidimensional for J: decreasing differences x mu} $ \phi(\bar{x},\bar{\mu})- \phi(x,\bar{\mu}) \leq \phi(\bar{x},\mu)-  \phi(x,\mu )$
for all $ {x}, \bar x \in \mathbb{R}^d$ and $ {\mu}, \bar \mu \in \mathcal{P}_2(\mathbb{R}^d)$ with $x \leq \bar{x} $ and $\mu \leq^{\text{\tiny{\rm st}}}  \bar \mu$. 
\end{enumerate}
\end{enumerate}
\end{assumption}

\begin{assumption}[Nonseparable Hamiltonian, affine drift and costs concave in the state]\label{assumption for comparison} \ 
\begin{enumerate}  
    \item \label{assumption structural with comparison} The coefficients of the SDE have the following structure
        \begin{enumerate} 
            \item $b(t,x,\mu,a) = b_0(t,\mu) + \bar b _1(t) x +  b_2 (t) a$ for functions 
\begin{align*}
    b_0 &: [0,T] \times \mathcal P _2 (\R^d)  \rightarrow  \R^d, \\
    \bar b _1 &: \, [0,T]  \rightarrow  \R^{d\times d}, \quad  \bar b _1(t) = ((\bar b _1(t))^i_j)_{\text{\tiny{$j=1,...,d$}}}^{\text{\tiny{$i=1,...,d$}}}, \\
    b_2 &: \, [0,T]  \rightarrow  \R^{d \times k},  \quad b _2 (t) = (( b _2 (t))^i_ \ell )_{\text{\tiny{$\ell =1,...,k$}}}^{\text{\tiny{$i=1,...,d$}}}; 
\end{align*}
        \item \label{assumption monotonicity multidimensional with comparison} For $dt$ a.a.\ $t \in [0,T]$, one has $b_0(t,\mu) \leq b_0(t,\bar \mu)$ for any $\mu, \bar \mu \in \mathcal P _2 (\R^d)$ with $\mu \leq^{\text{\tiny{\rm st}}}  \bar \mu$ and  $(\bar b _1(t))^i_j, (b_2(t))^i_\ell \geq 0 $ for any $i,j=1,...,d$, $\ell = 1,...,k$;
            \item $\sigma^{i}= (\sigma^i_j)_{j=1,...,d_1}$ and $\sigma^{i,\circ}= (\sigma^{i,\circ}_j)_{j=1,...,d_2}$ only depend on $x^i$.
        \end{enumerate} 
    \item\label{assumption for comparison: growth condition Dh Dg} (Growth conditions)  $|D_x h (t,x,\mu,a)| + |D_x g (x,\mu)| \leq K(1+\| \mu \|_1 )$, for any $(t,x,\mu, a) \in [0,T] \times \R ^d \times \mathcal P _2(\R^d) \times A$.  
    \item \label{assumption submodularity multidimensional with comparison} (Submodularity conditions) 
        For $dt$ a.a.\ $t \in [0,T]$,  we have
         \begin{enumerate}
                \item\label{assumption submodularity multidimensional with comparison.h submodular a} $h(t,x,\mu,a\land \bar a ) + h(t,x,\mu,a\lor \bar a ) \leq h(t,x,\mu,a ) +h(t,x,\mu,\bar a )$  for all $ x \in \mathbb{R}^d$, $ \mu \in \mathcal{P}_2(\mathbb{R}^d)$ and $a,\bar a \in A$;
                \item\label{assumption submodularity multidimensional with comparison.h decreasing differences} $h(t, \bar x,\bar \mu, \bar a ) - h(t, \bar x,\bar \mu, a ) \leq h(t,x, \mu, \bar a ) +h(t,x,\mu, a )$  for all $ x, \bar x \in \mathbb{R}^d$ with $x \leq \bar x$, $ \mu, \bar \mu \in \mathcal{P}_2(\mathbb{R}^d)$ with $\mu \leq^{\text{\tiny{\rm st}}}  \bar \mu$ and $a,\bar a \in A$ with $a\leq \bar a$;
                \item\label{assumption submodularity multidimensional with comparison.h for using comparison principle} $D_x h(t, \bar x,\bar \mu, \bar a ) \leq D_xh(t,  x, \mu,  a )$ and $D_xg( \bar x,\bar \mu ) \leq D_xg(  x, \mu)$ for all $ x, \bar x \in \mathbb{R}^d$ with $x \leq \bar x$, $ \mu, \bar \mu \in \mathcal{P}_2(\mathbb{R}^d)$ with $\mu \leq^{\text{\tiny{\rm st}}}  \bar \mu$ and $a,\bar a \in A$ with $a\leq \bar a$.
        \end{enumerate}
\end{enumerate}
\end{assumption}
A function $\phi$ satisfying Condition \ref{assumption submodularity multidimensional for J: submodular in x} in Assumption \ref{assumption for J} is said to be sumbodular in $x$, while a function $\phi$ satisfying Condition \ref{assumption submodularity multidimensional for J: decreasing differences x mu} in Assumption \ref{assumption for J} is said to have decreasing differences in $x$ and $\mu$.
Using this terminology, Condition \ref{assumption submodularity multidimensional with comparison.h submodular a} in Assumption \ref{assumption for comparison} means that the function 
$h(t,x,\mu, \cdot)$ is submodular in $a$, while Condition \ref{assumption submodularity multidimensional with comparison.h decreasing differences} that the function 
$h(t,\cdot,\cdot, \cdot)$ have decreasing differences in $a$ and $(x,\mu)$. 

The conditions above are particularly easy to check for \emph{mean-field interactions of scalar type} and \emph{mean-field interactions of order-1}.
Examples are discussed in Subsection \ref{subsection Checking the submodularity conditions}.

\begin{remark}
    \label{remark Topkis}   
    Conditions \ref{assumption for J} and \ref{assumption for comparison} represent a natural counterpart, for stochastic differential MFGs, of the submodularity conditions introduced in \cite{To} for static $N$-player games. 
    For this reason, when either Condition \ref{assumption for J} or \ref{assumption for comparison} is satisfied, we say the MFG to be submodular.
    Similar (but less general) conditions have been used in \cite{dianetti.ferrari.fischer.nendel.2019} (see also  
    \cite{Adlakhaetal, CarmonaDelarueLacker16, dianetti.ferrari.fischer.nendel.2022unifying}).    
\end{remark}

\begin{remark}[Submodularity and Lasry-Lions monotonicity condition]
    \label{remark Lasry-Lions monotonicity condition}
A well known structural condition which ensures uniqueness of equilibria in MFGs is  the --so called-- \emph{Lasry-Lions monotonicity condition}, which is typically required for separable Hamiltonians.  
In the notation of Condition \ref{assumption submodularity multidimensional for J} in Assumption \ref{assumption structural for J}, the Lasry-Lions monotonicity condition is satisfied if
\begin{equation}
\label{ass.monotonicity.Lasry.Lion}
    \int_{\R^d} (\phi(x, \bar{\mu}) - \phi(x, {\mu}))d(\bar{\mu}-\mu)(x) \geq 0, \quad \text{for any } \mu,  \bar \mu \in \mathcal{P}_2({\R^d}).
\end{equation}

Instead, when Condition \ref{assumption submodularity multidimensional for J} in Assumption \ref{assumption structural for J} is verified, then the map  $\phi(\cdot, \bar\mu)-\phi(\cdot, \mu)$ is nonincreasing for $\mu, \bar\mu \in \mathcal P_2 (\R^d)$ with $\mu \leq^{\text{\tiny{\rm st}}} \bar \mu$, and we obtain  
\begin{equation*}
    \int_{\R^d} (\phi(x, \bar{\mu}) - \phi(x, {\mu}))d(\bar{\mu}-\mu)(x) \leq 0, \quad \text{for any } \mu,  \bar \mu  \in \mathcal{P}_2 ({\R^d}) \text{ with } \mu \leq^{\text{\tiny{\rm st}}} \bar\mu.
\end{equation*}
The latter represents, roughly speaking, an opposite version of condition \eqref{ass.monotonicity.Lasry.Lion}. 
\end{remark}     


Either of the previous sets of conditions implies the monotonicity of the best reply maps. 
This property is formalized in the following proposition, which is proved in Subsections \ref{subsection proof proposition best reply map increasing: cost functional} and \ref{subsection proof best reply increasing: comparison}.  

\begin{proposition}[Monotonicity of best reply maps]
    \label{proposition best reply map increasing}
Under Assumptions \ref{assumption existence}, \ref{assumption uniqueness FBSDEs} and either Assumption \ref{assumption for J} or \ref{assumption for comparison}, 
the following monotonicity properties hold true:
\begin{enumerate}
    \item\label{proposition best reply map increasing: claim Gamma} The  map $\Gamma$ is increasing; i.e., $\Gamma( \mu) \leq \Gamma (\bar \mu)$ for any $\mu, \, \bar \mu \in  \mathcal M _{\text{\tiny{$B$}}}^2$ with $\mu \leq^{\text{\tiny{\rm st}}}  \bar \mu$;
    \item The map $\mathcal R$ is increasing; i.e., $\mathcal R( \mu) \leq^{\text{\tiny{\rm st}}}  \mathcal R (\bar \mu)$ for any $\mu, \, \bar \mu \in  \mathcal M _{\text{\tiny{$B$}}}^2$ with $\mu \leq^{\text{\tiny{\rm st}}}  \bar \mu$; 
    \item  The map $R$ is increasing; i.e., $R (M) \leq R(\bar M)$ for any $M, \, \bar M \in \mathbb H ^2$ with $M \leq \bar M$. 
\end{enumerate}
\end{proposition}

We observe that, in light of Lemma \ref{lemma existence of optimal controls}, Proposition \ref{proposition best reply map increasing} can be restated in terms of a comparison principle for the forward component of the solution of the FBSDE \eqref{eq FBSDE optimal controls}.
\begin{corollary}[Comparison principle]
    \label{corollary comparison principle}
Under Assumptions \ref{assumption existence}, \ref{assumption uniqueness FBSDEs} and either Assumption \ref{assumption for J} or \ref{assumption for comparison}, let $\mu, \, \bar \mu \in  \mathcal M _{\text{\tiny{$B$}}}^2$ and let $(X,Y,Z, Z^\circ)$, $(\bar X, \bar Y, \bar Z, \bar Z ^\circ)$ be the solutions to the FBSDE \eqref{eq FBSDE optimal controls} with $\mu$, $\bar \mu$, respectively.  
If $\mu \leq^{\text{\tiny{\rm st}}}  \bar \mu$, then $X \leq \bar X$.
\end{corollary} 
\begin{remark}
    \label{remark comparison principle}
While the previous result is known under Assumption \ref{assumption for comparison}, it is genuinely new under Assumption \ref{assumption for J}. 
Indeed, under Assumption \ref{assumption for comparison}, Corollary \ref{corollary comparison principle} follows from a comparison principle for FBSDEs (Theorem 2.2 in \cite{chen.luo2021}), and such a result will actually be used in order to prove Proposition \ref{proposition best reply map increasing} (see Subsection \ref{subsection proof best reply increasing: comparison}). 
On the other hand, under Assumption \ref{assumption for J}, Corollary \ref{corollary comparison principle} covers relevant cases not covered by Theorem 2.2 in \cite{chen.luo2021},  as, for example, the cases in which $h$ and $g$ are convex in $x$ (notice indeed that some monotonicity assumptions in \cite{chen.luo2021} need to be inverted to be employed the present context). 

Comparison principles for multidimensional FBSDEs are difficult to obtain. 
Indeed,  the corresponding decoupling field is the solution of a parabolic system, for which comparison principles fail in general (see, e.g., \cite{cvitanic.ma.96} or Example 6.2 at p.\ 22 in \cite{ma.yong.99}). 
Corollary \ref{corollary comparison principle} offers instead a new natural point of view on this problem, and it connects a comparison principle to properties of the underlying optimization problem. 
Roughly speaking, for parameter-dependent-FBSDEs corresponding (via the stochastic maximum principle) to a minimization problem which is submodular in the parameter, a comparison principle for the forward component holds.
\end{remark}

\subsection{Lattice of processes and a priori estimates} 

Given processes $M, \, \bar M \in \mathbb H ^2$, set 
\begin{equation*}
    M \leq \bar M \text{ if and only if } M_t \leq \bar M _t \ \P \otimes \pi\text{-a.e.\ in $\Omega \times [0,T]$}, 
\end{equation*}
and define a lattice structure on $\mathbb H ^2$ by setting
\begin{equation}
    \label{eq definition land lor processes}
    (M\land \bar M)_t := M_t \land \bar M _t \quad \text{and} \quad (M\lor \bar M)_t := M_t \lor \bar M _t, \quad \mathbb P \otimes \pi\text{-a.e.\ in $\Omega \times [0,T]$.}
\end{equation}
The set $(\mathbb H ^2, \leq)$ is a partially ordered set, and the operations $\land, \, \lor$ provide a lattice structure on $\mathbb H ^2$, which is compatible with the order relation $\leq$.
Similarly, for  processes $\mu, \, \bar \mu \in\mathcal M _{\text{\tiny{$B$}}}^2$, set 
\begin{equation*}
    \mu \leq^{\text{\tiny{\rm st}}}  \bar \mu \text{ if and only if } \mu _t \leq^{\text{\tiny{\rm st}}}  \mu _t \ \P \otimes \pi\text{-a.e.\ in $\Omega \times [0,T]$}.  
\end{equation*}
The set $(\mathcal M _{\text{\tiny{$B$}}}^2,\leq^{\text{\tiny{\rm st}}} ) $ is a partially ordered set, but it does not have a lattice structure when $d>1$ (see Remark \ref{remark multidimensional} below). 
Moreover, from \eqref{eq monotonicity of conditional distribution}, we have that the projection map $p:\mathbb H ^2 \to \mathcal M _{\text{\tiny{$B$}}}^2 $ (see Remark \ref{remark characterization of MFGE}) is incresing; that is
\begin{equation}
    \label{eq projection monotone} 
    p(M) \leq^{\text{\tiny{\rm st}}} p(\bar M) \quad \text{if} \quad M \leq \bar M, \quad \text{for any } M,\bar M \in \mathbb H ^2. 
\end{equation} 

Taking $\essinf$ and $\esssup$ with respect to the measure $\P \otimes \pi$, define the processes
\begin{equation}\label{eq definition Gamma upper lower bounds}
\underline{\Gamma} := \essinf \{ R(M)| M \in \bH^2 \}  \quad \text{and} \quad 
\overline{\Gamma} := \esssup \{ R(M)| M \in \bH^2 \}. 
\end{equation}

In order to ensure the square integrability of the processes $\underline \Gamma$ and $\overline \Gamma$, we impose the following requirements. 
\begin{assumption}\label{assumption a priori estimates}
The drift $b_1$ satisfies $b_1(t,x,\mu) \leq K(1+|x|)$ for any $(t,x,\mu) \in [0,T] \times \R ^d \times \mathcal P _2 (\R^d)$, 
 and either of the following conditions is satisfied:
\begin{enumerate}
    \item\label{assumption a priori estimates A compact} $A$ is compact;
    \item\label{assumption a priori estimates coercivity} For  $K$ as in Assumption \ref{assumption existence}, the functions $h,g$ satisfy 
    \begin{align*} 
    &  g (x,\mu)\leq K(1+|x|^2), \\ 
    &  h(t,x,\mu,a) \leq K(1+|x|^2 + |a|^2),
    \end{align*}
    for any $(t,x,\mu, a) \in [0,T] \times \R ^d \times \mathcal P _2(\R^d) \times A$.  
\end{enumerate}
\end{assumption}

From Assumption \ref{assumption a priori estimates} and Proposition \ref{proposition best reply map increasing} we  obtain the following result, which is proved in Subsection \ref{subsection proof a priori estimate}.
\begin{lemma}[A priori estimates]\label{lemma a priori estimates}
Under Assumptions \ref{assumption existence}, \ref{assumption uniqueness FBSDEs}, \ref{assumption a priori estimates} and either Assumption \ref{assumption for J} or \ref{assumption for comparison}, we have
\begin{equation}\label{eq moment a priori estimates}
\sup_{\mu \in \mathcal M _{\text{\tiny{$B$}}}^2} \mathbb E \bigg[ \sup_{t \in [0,T]} |X_t^\mu|^2 \bigg] < \infty.
\end{equation}
Moreover, $\underline{\Gamma}, \overline{\Gamma} \in \bH^2$.
\end{lemma}

In particular, it is worth underlining that the fact that $\underline{\Gamma}, \overline{\Gamma} \in \bH^2$ is a consequence of the monotonicity of the best reply map (see Proposition \ref{proposition best reply map increasing}).

\subsection{Main results: Existence and approximation of solutions}
We are now ready to discuss the main results of this paper. 

\subsubsection{Solutions to the MKV FBSDE}
Define inductively two sequences of processes as follows:
\begin{enumerate}
    \item Set $\underline X ^0 := \underline \Gamma$ and, for $n \geq 1$, define $(\underline X ^{n}, \underline Y ^{n}, \underline Z ^{n}, \underline Z ^{\circ,{n}})$ as the solution to the FBSDE
\begin{align*}
    & d \underline X _t ^{n} = \hat b ( t, \underline X _t^{n}, \mathbb P _{\text{\tiny{$\underline X ^{n-1}_t|B$}}} , \underline Y _t^{n} ) dt + \sigma (t, \underline X_t^{n}) d W_t + \sigma^\circ (t, \underline X_t^{n}) d B_t, \quad \underline X _0^{n} = \xi, \\  \notag
    & d \underline Y _t^{n} = - \hat h (t,\underline X _t^{n},\mathbb P _{\text{\tiny{$\underline X ^{n-1}_t|B$}}}, \underline Y _t^{n}) dt + \underline Z _t^{n} dW_t + \underline Z _t^{\circ,{n}} dB_t, \quad \underline Y _T^{n} = \hat g( \underline X _T^{n}, \mathbb P _{\text{\tiny{$\underline X ^{n-1}_T|B$}}}); \notag
\end{align*}
\item Set $\overline X ^0 := \overline \Gamma$ and, for $n \geq 1$, define $(\overline X ^{n}, \overline Y ^{n}, \overline Z ^{n}, \overline Z ^{\circ,{n}})$ as the solution to the FBSDE
\begin{align*}
    & d \overline X _t ^{n} = \hat b ( t, \overline X _t^{n}, \mathbb P _{\text{\tiny{$\overline X ^{n-1}_t|B$}}} , \overline Y _t^{n} ) dt + \sigma (t, \overline X_t^{n}) d W_t + \sigma^\circ (t, \overline X_t^{n}) d B_t, \quad \overline X _0^{n} = \xi, \\  \notag
    & d \overline Y _t^{n} = - \hat h (t,\overline X _t^{n},\mathbb P _{\text{\tiny{$\overline X ^{n-1}_t|B$}}}, \overline Y _t^{n}) dt + \overline Z _t^{n} dW_t + \overline Z _t^{\circ,{n}} dB_t, \quad \overline Y _T^{n} = \hat g( \overline X _T^{n}, \mathbb P _{\text{\tiny{$\overline X ^{n-1}_T|B$}}}). \notag 
\end{align*}
\end{enumerate}
In particular, notice that, in light of Lemma \ref{lemma existence of optimal controls} and of \eqref{eq definition best reply maps}, the sequences $(\underline X ^n)_n$ and $(\overline X ^n)_n$ are obtained by iterating the best reply map $R$; that is, 
\begin{equation}
\label{eq learning procedure X}
\underline X ^0 = \underline \Gamma, \ \underline  X ^{n+1} =  R (\underline X ^{n}), \ n \in \mathbb N
\quad \text{and} \quad 
\overline X ^0 = \overline \Gamma, \ \overline X ^{n+1} =  R (\overline X ^{n}), \ n \in \mathbb N.
\end{equation}

The convergence of the sequences  
$(\underline X ^{n}, \underline Y ^{n}, \underline Z ^{n}, \underline Z ^{\circ,{n}})_n, \,(\overline X ^{n}, , \overline Y ^{n}, \overline Z ^{n}, \overline Z ^{\circ,{n}})_n $ is subject to the following continuity condition. 
\begin{assumption}\label{assumption continuity in the measure}
  The functions $h(t, \cdot, \cdot, \cdot),\, D_x h(t, \cdot, \cdot, \cdot),\, g,$ and $ D_x g$ are jointly continuous in $(x,\mu,a)$ on $\R^d \times \mathcal P _2(\R^d) \times A$, where on $\mathcal P _2(\R^d)$ we consider the $\ell_{\R^d,2}$ Wasserstein distance (see \eqref{eq Wasserstain distance}).
\end{assumption}

\begin{remark}[Continuity of $\hat \alpha$]
    \label{remark continuity feedback in mu}
    Following the discussion before equation (3.57) at p.\ 209 in Vol.\ II of \cite{CarmonaDelarue18}, it is possible to show that, under Assumption \ref{assumption existence} and \ref{assumption continuity in the measure}, the feedback $\hat \alpha$ is continuous in $(x,\mu,y)$, for any $t \in [0,T]$.
\end{remark}

Introduce an order relation on the set of the solutions of the MKV FBSDE.  
Given two solutions $(X,Y,Z, Z^\circ)$ and $(\bar X, \bar Y, \bar Z, \bar Z ^\circ)$, if $X \leq \bar X$,  set $(X,Y,Z, Z^\circ) \leq (\bar X, \bar Y, \bar Z, \bar Z ^\circ)$.
With such an order relation, the set of solutions of the MKV FBSDE is a partially ordered set. 
 
The following theorem is the first main result of this paper and it is proved in Subsection \ref{subsection proof first main result}.  
\begin{theorem}\label{theorem main FBSDE} Under Assumptions \ref{assumption existence}, \ref{assumption uniqueness FBSDEs}, \ref{assumption a priori estimates} and either Assumption \ref{assumption for J} or \ref{assumption for comparison},  the following statements hold true:
 \begin{enumerate}
  \item\label{theorem.main.existence.FBSDE} The set of solutions to the MKV FBSDE \eqref{eq MKV FBSDE} is a nonempty complete lattice (compatible with the order relation $\leq$): 
  in particular, there exist minimal and maximal solutions $(\underline X, \underline Y, \underline Z, \underline Z ^\circ)$ and $(\overline X, \overline Y, \overline Z, 
 \overline Z ^\circ)$, respectively. 
  \item If the additional Assumption \ref{assumption continuity in the measure} holds, then: 
  \begin{enumerate}\label{theorem.main.general.convergence.FBSDE}
  \item\label{theorem.main.general.convergence.up.FBSDE} The sequence $(\underline{X}^n)_n$ is monotone increasing, it converges to $\underline X$, $\pi \otimes \mathbb P$-a.e., and 
  $$
\lim_n \E \bigg[ \sup_{ t \in [0,T]} (|\underline{X}_t^n - \underline X _t|^2 + |\underline{Y}_t^n - \underline Y _t|^2) 
    + \int_0^T ( | \underline{Z}_t^n - \underline Z _t |^2+ | \underline{Z} _t^{\circ,n} - \underline Z _t^\circ |^2 )dt \bigg] =0;
$$
  \item\label{theorem.main.generalconvergence.down.FBSDE} The sequence $(\overline {X}^n)_n$ is monotone decreasing, it converges to $\overline X$, $\pi \otimes \mathbb P$-a.e., and   
  $$
\lim_n \E \bigg[ \sup_{ t \in [0,T]} (|\overline {X}_t^n - \overline  X _t|^2 + |\overline {Y}_t^n - \overline  Y _t|^2) 
    + \int_0^T ( | \overline {Z}_t^n - \overline  Z _t |^2+ | \overline {Z} _t^{\circ,n} - \overline  Z _t^\circ |^2 )dt \bigg] =0. 
$$
\end{enumerate}
\end{enumerate}
\end{theorem}

\subsubsection{MFG equilibria} 
We next interpret Theorem \ref{theorem main FBSDE} in terms of the MFG problem. 
Set
\begin{equation*}
\mu^{\rm Min} : = (  \mathbb P _{\text{\tiny{$\underline{\Gamma}_t |B$}}})_{t \in [0,T]}
\quad \text{and} \quad 
\mu^{\rm Max} : = (  \mathbb P _{\text{\tiny{$ \overline {\Gamma}_t |B$}}})_{t \in [0,T]},
\end{equation*}  
and define inductively two sequences of processes $(\underline \mu ^{n})_n, \,(\overline \mu ^{n})_n \subset \mathcal M _{\text{\tiny{$B$}}}^2$ as follows: 
\begin{equation}
    \label{eq learning procedure mu}
\underline \mu ^0 := \mu^{\rm Min}, \ \underline \mu ^{n+1} := \mathcal R (\underline \mu ^{n}),  \ n \in \mathbb N
\quad \text{and} \quad 
\overline \mu ^0 := \mu^{\rm Max}, \ \overline \mu ^{n+1} := \mathcal R (\overline \mu ^{n}), \ n \in \mathbb N.
\end{equation} 
Notice that $\underline \mu ^n =( \mathbb P _{\text{\tiny{$\underline X ^{n}_t|B$}}})_{t\in [0,T]}$ and $\overline \mu ^n =( \mathbb P _{\text{\tiny{$\overline X ^{n}_t|B$}}})_{t\in [0,T]}$.
In order to discuss the convergence of the sequences  $(\underline \mu ^{n})_n, \,(\overline \mu ^{n})_n$, introduce the set $\mathcal C ^d$ of continuous functions $\varphi: [0,T] \to \R^d$, equipped with the sup norm. 
On $\mathcal P_2 (\mathcal C^d)$, consider the $2$-Wasserstein distance $\ell_{\mathcal C ^d, 2}$ (see \eqref{eq Wasserstain distance}).

The following theorem is the second main result of this paper. 
Its proof is given in Subsection \ref{subsection proof second main result}.
\begin{theorem}\label{theorem main} Under Assumptions \ref{assumption existence}, \ref{assumption uniqueness FBSDEs}, \ref{assumption a priori estimates} and either Assumption \ref{assumption for J} or \ref{assumption for comparison},  the following statements hold true:
 \begin{enumerate}
  \item\label{theorem.main.existence} The set of MFG equilibria is a nonempty complete lattice (compatible with the order relation $\leq^{\text{\tiny{\rm st}}} $):  
  in particular, there exist minimal and maximal solutions $\underline \mu$ and $\overline \mu$, respectively. 
  \item If the additional Assumption \ref{assumption continuity in the measure} holds, then:
  \begin{enumerate}\label{theorem.main.general.convergence}
  \item\label{theorem.main.general.convergence.up} The sequence $(\underline {\mu}^n)_n$ is such that 
  $\underline {\mu}^{n} \leq^{\text{\tiny{\rm st}}} \underline {\mu}^{n+1}$ for any $n \in \mathbb N$ and 
  $$
  \lim_n \ell_{\mathcal C ^d, 2} (\underline {\mu}^n, \underline{\mu}) =0, \ \text{ $\mathbb P$-a.s.,} \quad  \text{and} \quad \lim_n \mathbb E [ \ell_{\mathcal C ^d, 2}^2 (\underline {\mu}^n, \underline {\mu})] =0;  
  $$
  \item\label{theorem.main.generalconvergence.down} The sequence $(\overline {\mu}^n)_n$ is such that 
  $\overline {\mu}^{n+1} \leq^{\text{\tiny{\rm st}}} \overline {\mu}^{n}$ for any $n \in \mathbb N$ and 
  $$
  \lim_n \ell_{\mathcal C ^d, 2} (\overline {\mu}^n, \overline{\mu}) =0, \ \text{ $\mathbb P$-a.s.,} \quad  \text{and} \quad \lim_n \mathbb E [ \ell_{\mathcal C ^d, 2}^2 (\overline {\mu}^n, \overline{\mu})] =0.
  $$
 \end{enumerate}
  \end{enumerate}
\end{theorem} 

\subsubsection{Convergence of the fictitious play} 
We now study the convergence of the learning procedure introduced (in the context of MFGs) in \cite{CardaliaguetHadikhanloo17}. 
Take $\hat X ^1 \in \mathbb H^2, \ \hat \mu ^{1}= ( \mathbb P _{\text{\tiny{$\hat {X}_t^1 |B$}}})_{t\in [0,T]}$ and, for $n > 1$, define $(\hat X ^{n}, \hat Y ^{n}, \hat Z ^{n}, \hat Z ^{\circ,{n}})$ as the solution to the FBSDE
\begin{align*}
    & d \hat X _t ^{n} = \hat b ( t, \hat X _t^{n}, \hat \mu _t^{n-1} , \hat Y _t^{n} ) dt + \sigma (t, \hat X_t^{n}) d W_t + \sigma^\circ (t, \hat X_t^{n}) d B_t, \quad \hat X _0^{n} = \xi, \\  \notag
    & d \hat Y _t^{n} = - \hat h (t,\hat X _t^{n},\hat \mu _t^{n-1}, \hat Y _t^{n}) dt + \hat Z _t^{n} dW_t + \hat Z _t^{\circ,{n}} dB_t, \quad \hat Y _T^{n} = \hat g( \hat X _T^{n}, \hat \mu _T^{n-1}), \notag
\end{align*}   
where $\hat \mu^n$ is defined by  
$$
\hat \mu _t^{n} : = \frac{1}{n} \sum_{k=1}^n \mathbb P _{\text{\tiny{$\hat {X}_t^k |B$}}}. 
$$
Notice that, in light of Lemma \ref{lemma existence of optimal controls} and of \eqref{eq definition best reply maps}, the sequences $(\hat X ^n)_n$ and $(\hat \mu ^n)_n$ ca be written in terms of the best reply maps $\Gamma$ and $\mathcal R$; that is, 
\begin{equation}
\label{eq fictitious play}
\hat X ^{n+1} =  \Gamma (\hat \mu ^{n}), \ n \in \mathbb N _0
\quad \text{and} \quad 
\hat \mu ^{n+1} =  \frac{1}{n+1} (\mathcal R (\hat \mu^n) + n \hat \mu^n), \ n \in \mathbb N _0.
\end{equation}
The sequence $(\hat \mu ^n)_n$ is referred to as fictitious play. 
 
With  $(\underline X, \underline Y, \underline Z, \underline Z ^\circ)$ and $\underline \mu$ as in Theorems \ref{theorem main FBSDE} and \ref{theorem main}, we have the following convergence result. 
Its proof follows arguments similar to those in the proof of Theorems \ref{theorem main FBSDE} and \ref{theorem main} (see Section \ref{section proof main results}), and it is therefore only sketched in Subsection \ref{subsection proof convergence fictitious play}. 

\begin{theorem}\label{theorem main fictitious play} 
Under Assumptions \ref{assumption existence}, \ref{assumption uniqueness FBSDEs}, \ref{assumption a priori estimates}, \ref{assumption continuity in the measure} and either Assumption \ref{assumption for J} or \ref{assumption for comparison}, if $\hat X ^1 =\underline \Gamma$, then the following statements hold true:
 \begin{enumerate}
  \item\label{theorem main fictitious play FBSDE} The sequence $(\hat{X}^n)_n$ is monotone increasing, it converges to $\underline X$, $\pi \otimes \mathbb P$-a.e., and 
  $$
\lim_n \E \bigg[ \sup_{ t \in [0,T]} (|\hat{X}_t^n - \underline X _t|^2 + |\hat{Y}_t^n - \underline Y _t|^2) 
    + \int_0^T ( | \hat{Z}_t^n - \underline Z _t |^2+ | \hat{Z} _t^{\circ,n} - \underline Z _t^\circ |^2 )dt \bigg] =0;
$$
  \item\label{theorem main fictitious play measures} The sequence $(\hat {\mu}^n)_n$ is such that 
  $\hat {\mu}^{n} \leq^{\text{\tiny{\rm st}}} \hat {\mu}^{n+1}$ for any $n \in \mathbb N _0$ and 
  $$ 
  \lim_n \ell_{\mathcal C ^d, 2} (\hat {\mu}^n, \underline{\mu}) =0, \ \text{ $\mathbb P$-a.s.,} \quad  \text{and} \quad \lim_n \mathbb E [ \ell_{\mathcal C ^d, 2}^2 (\hat {\mu}^n, \underline {\mu})] =0. 
  $$
\end{enumerate}
\end{theorem}
Clearly, analogous convergence results to the maximal solutions  $(\overline X, \overline Y, \overline Z, \overline Z ^\circ)$ and $\overline \mu$ hold if $\hat X ^0 =\overline \Gamma$. 

\subsubsection{Some remarks} 
This subsection collects some observations on Theorems \ref{theorem main FBSDE}, \ref{theorem main} and \ref{theorem main fictitious play}.

\begin{remark}[Discontinuity in the measure $\mu$]
It is worth observing that Claim \ref{theorem.main.existence.FBSDE} in Theorem \ref{theorem main FBSDE} establishes the existence of minimal and maximal strong solutions to the conditional MKV FBSDE \eqref{eq MKV FBSDE} through a lattice theoretical fixed point theorem (see Subsection \ref{subsection proof first main result}).
This approach is very different from the ones already used in the literature. 
In particular, it does not require the time horizon to be small and the Lipschitz continuity in the measure flow (see \cite{CarmonaDelarue18}), or the use of the --so called-- monotonicity conditions (see \cite{ahuja2016wellposedness,ahuja.ren.yang.2019, huang.tang.2021}). 
\end{remark} 

\begin{remark}[Explicit initialization of the algorithms] 
The initialization  of the sequences in \eqref{eq learning procedure X}, \eqref{eq learning procedure mu} and in \eqref{eq fictitious play} can be made more explicit if the control set $A$ is assumed to be compact. 
 
For example, by the monotonicity of the best reply maps, the conclusions of Claim \ref{theorem.main.general.convergence.up.FBSDE} in Theorem \ref{theorem main FBSDE} can be obtain for any sequence $(M^n)_n$ of type 
$M^0 \in \mathbb H ^2$, $ M^{n+1} := R(M^n)$ for $ n \in \mathbb N$,
whenever $M^0$ verifies $M ^0 \leq \underline \Gamma$. 
In particular, when the assumptions of Theorem \ref{theorem main FBSDE} are satisfied and the control set $A$ is compact, one can choose $M^0$ as the solution of the SDE 
\begin{equation*}
dM_t^0= (\tilde{b}_1(t,M_t^0) +  \tilde{b}_2(t))dt + \sigma(t,M_t^0) dW_t + \sigma^{\circ} (t,M_t^0) dB_t,  \quad t \in [0,T], \quad M_0^0=\xi,  
\end{equation*}
where 
$$
\text{$\tilde{b}_1(t,x) := \inf_{\mu \in \mathcal  M _{\text{\tiny{$B$}}}^2 } b_1(t,x,\mu_t) \quad $and$ \quad \tilde{b}_2(t) := \inf_{a \in A} b_2(t)a$.}
$$
By the comparison principle for SDEs (see, e.g., \cite{Protter05}), one obtains $M^0 \leq X^{\alpha,\mu}$, for any $\alpha \in \mathcal A$ and $\mu \in \mathcal  M _{\text{\tiny{$B$}}}^2$, where $X^{\alpha,\mu}$ denotes the solution of \eqref{eq controlled SDE}. 
Hence, by the definition of $\underline \Gamma$ in \eqref{eq definition Gamma upper lower bounds}, one has $M^0 \leq  \underline \Gamma$,  so that the convergence of $M^n$ to $\underline X$ holds. 

Obviously, analogous considerations can be made also for the initialization of the sequences $(\overline X ^n)_n$ and $\hat X ^n$. 
\end{remark} 

\begin{remark} 
    \label{remark multidimensional}
In comparison with \cite{dianetti.ferrari.fischer.nendel.2019}, the approach used in this paper allows to treat multidimensional settings.  
In particular,  the fixed point problem in \cite{dianetti.ferrari.fischer.nendel.2019} was formulated on the lattice related to the partially ordered set $(\mathcal P (\R), \leq^{\text{\tiny{\rm st}}})$.
However, although the first order stochastic dominance induces a lattice structure on $\mathcal{P}(\mathbb{R})$, it does not induce a lattice order on $\mathcal{P}(\mathbb{R}^d)$ for $d>1$ (cf.\ \cite{kamae&obrien} and \cite{muller&scarsini}).
Thus, Tarski's fixed point theorem (see \cite{T}) can not be employed on the set $(\mathcal P (\R^d), \leq^{\text{\tiny{\rm st}}})$.
This technical limitation is overcome in the present work by formulating the fixed point problem on the lattice of processes related to the partially ordered set $(\mathbb H^2, \leq )$.
Indeed,  while the set $\mathcal P (\R^d)$ has good compactness properties
and is a natural choice in order to use topological fixed point theorems, the set $\mathbb H ^2$ has a more natural lattice structure and represents a better choice when using Tarski's fixed point theorem. 
\end{remark}

\section{Proofs of Proposition \ref{proposition best reply map increasing} and of Lemma \ref{lemma a priori estimates}}
    \label{section proof auxiliary results}
\subsection{Proof of Proposition \ref{proposition best reply map increasing} under Assumption \ref{assumption for J}}
    \label{subsection proof proposition best reply map increasing: cost functional}
We recall that under Assumption \ref{assumption for J} the drift $b_1$ does not depend on the measure argument $\nu$.  
In this subsection, for $\alpha \in \mathcal A$, we denote by
$X^\alpha$ the solution to the SDE$_\alpha$ \eqref{eq controlled SDE}. 

\subsubsection{A preliminary lemma}  
We first prove that the set of controlled trajectories is a lattice. 
\begin{lemma}\label{lemma lattice of trajectories} 
For $\alpha, \, \bar{\alpha} \in \mathcal A$, the controls $\alpha^\land, \alpha^\lor \in \mathcal A$ defined by 
\begin{equation*}    
\alpha_t^{i,\land} := \alpha_t^i \mathds{1}_{ \{ X_t^{i,\alpha} < X^{i,\bar{\alpha}}_t \}} + \bar{\alpha}_t^i \mathds{1}_{ \{ X^{i,{\alpha}}_t \geq  X^{i,\bar{\alpha}}_t \}} 
\quad  \text{and} \quad 
\alpha_t^{i,\lor} := \bar{\alpha}_t^i \mathds{1}_{ \{ X^{i,{\alpha}}_t < X^{i,\bar{\alpha}}_t \}} + {\alpha}_t^i \mathds{1}_{ \{ X^{i,{\alpha}}_t \geq X^{i,\bar{\alpha}}_t \}},  
\end{equation*} 
are such that $X^\alpha \land X^{\bar{\alpha}} = X^{\alpha^\land}$ and $X^\alpha \lor X^{\bar{\alpha}} = X^{\alpha^\lor}$.  
\end{lemma}

\begin{proof}
Take $\alpha, \, \bar{\alpha} \in \mathcal A$ and set $X:= X^\alpha, \, \bar X := X^{\bar{\alpha}}$. 
The controls $\alpha^\land, \, \alpha^\lor$ are clearly admissible.
For any $i=1,...,d$, by the Meyer-It\^o formula for continuous semimartingales (see, e.g., Theorem 68 at p.\ 216 in \cite{Protter05}), one has  
\begin{align*}
X_t^i \land \bar{X}_t^i &= X_t^i + 0 \land ( \bar{X}_t^i - X_t^i ) \\
& = X_t^i + \int_0^t  \mathds{1}_{ \{\bar{X}_s^i -X_s^i \leq 0 \}} d(\bar{X}^i -X^i)_s  - \frac{1}{2}L_t^0(\bar{X}^i -X^i)\\ 
& = \xi^i   +  \int_0^t \Big( \mathds{1}_{ \{X_s^i <\bar{X}_s^i \} } \Big( b_1^i(s,X_s^i) + b_2^i(s) \alpha^i _s \Big)  + \mathds{1}_{ \{X_s^i \geq \bar{X}_s^i \} } \Big( b_1^i(s,\bar{X}_s^i) + b_2^i(s) \bar{\alpha}_s^i \Big) \Big) ds
 \\ 
&\quad  + \int_0^t \Big( \mathds{1}_{ \{X_s^i <\bar{X}_s^i \} } \sigma^i(s,X_s^i) + \mathds{1}_{ \{X_s^i \geq \bar{X}_s^i \} } \sigma^{i} (s,\bar{X}_s^i) \Big) dW_s \\
&\quad  + \int_0^t \Big( \mathds{1}_{ \{X_s^i <\bar{X}_s^i \} } \sigma^{i,\circ}(s,X_s^i) + \mathds{1}_{ \{X_s^i \geq \bar{X}_s^i \} } \sigma^{i,\circ} (s,\bar{X}_s^i) \Big) d B_s \\
& \quad - \frac{1}{2}L_t^0(\bar{X}^i -X^i),
\end{align*}  
where $L_t^0(\bar{X}^i -X^i)$ is the local time of the process $\bar{X}^i-X^i$ at 0 (see, e.g., Chapter IV in \cite{Protter05}).
Denote by $[\bar{X}^i - X^i,\bar{X}^i - X^i ]$ the quadratic variation of the process $\bar{X}^i  -X^i$ (see, e.g., p. 66 in \cite{Protter05}). 
Using the characterization of local times (see, e.g., Corollary 3 at p.\ 230 in \cite{Protter05}), we obtain
\begin{align*}
L_t^0 (\bar{X}^i -X^i) & = \lim_{\varepsilon \to 0} \frac{1}{\varepsilon} \int_0^t \mathds{1}_{\{ 0 \leq \bar{X}_s^i - X_s^i \leq \varepsilon \}} d[\bar{X}^i - X^i,\bar{X}^i - X^i ]_s \\
& = \lim_{\varepsilon \to 0} \frac{1}{\varepsilon} \int_0^t \mathds{1}_{\{ 0 \leq \bar{X}_s^i - X_s^i \leq \varepsilon \}}  |\sigma^i (s,\bar{X}_s^i) - \sigma^i (s,{X}_s^i)|^2  ds \\
& \quad  + \lim_{\varepsilon \to 0} \frac{1}{\varepsilon} \int_0^t \mathds{1}_{\{ 0 \leq \bar{X}_s^i - X_s^i \leq \varepsilon \}}  |\sigma^{i,\circ} (s,\bar{X}_s^i) - \sigma^{i,\circ} (s,{X}_s^i)|^2 ds \\
& \leq 2 K^2 \lim_{\varepsilon \to 0} \frac{1}{\varepsilon} \int_0^t \mathds{1}_{\{ 0 \leq \bar{X}_s^i - X_s^i \leq \varepsilon \}}  |\bar{X}_s^i- {X}_s^i|^2 ds
\\
& \leq 2 K^2 \lim_{\varepsilon \to 0} \frac{1}{\varepsilon} \int_0^t \mathds{1}_{\{ 0 \leq \bar{X}_s^i - X_s^i \leq \varepsilon \}}  \varepsilon^2 ds = 0. 
\end{align*}  
Combining together the latter two equalities and using the definition of $\alpha^\land$, we conclude that
$$                      
X_t^i \land \bar{X}_t^i = \xi^i + \int_0^t \Big( b_1^i(s,X_s^i \land \bar{X}_s^i) + b_2^i(s) \alpha_s^{i,\land} \Big)  ds + \int_0^t  \sigma^i(s,X_s^i \land \bar{X}_s^i) dW_s + \int_0^t  \sigma^{i,\circ} (s,X_s^i \land \bar{X}_s^i) d B_s. 
$$
Since the index $i$ is generic, this proves that $X \land \bar{X} = X^{\alpha^\land}$.

In the same way,  
the process $X \lor \bar{X} = \bar X + 0 \lor ( X -\bar X )$ solves the SDE controlled by $\alpha ^\lor$, completing the proof of the lemma. 
\end{proof}  

\subsubsection{Proof of Proposition \ref{proposition best reply map increasing}}
By monotonicity of the map $p$ (see \eqref{eq projection monotone}), it is sufficient to prove Claim \ref{proposition best reply map increasing: claim Gamma}.

Take $\mu, \, \bar \mu \in  \mathcal M _{\text{\tiny{$B$}}}^2$ with $\mu \leq^{\text{\tiny{\rm st}}}  \bar \mu$ and let $(\bar X, \bar \alpha )$ and $(X,\alpha)$ be the optimal pairs related to $\bar{\mu} $ and $\mu$, respectively.  
 
By the admissibility  of $\alpha^{\vee}$ and the optimality of $\bar \alpha$, thanks to Lemma \ref{lemma lattice of trajectories} we can write 
\begin{align}\label{J-J}
0 \leq J(\alpha^{\vee},\bar{\mu}) - J(\bar \alpha,{\bar{\mu}}) 
= &  \mathbb{E} \bigg[ \int_0^T \Big(f(t,X_t\lor  \bar X_t, {\bar{\mu}}_t)  - f(t,\bar X _t, {\bar{\mu}}_t) \Big) dt  \bigg] \\ \notag
& \quad \quad + \mathbb{E} \bigg[ \int_0^T \sum_{i=1}^d \Big( l^i(t,X_t^i\lor \bar X_t^i, \alpha_t^{i,\vee}) -l^i(t,\bar X_t^i, \bar \alpha _t^i) \Big) dt \bigg] \\ \notag
& \quad \quad \quad \quad + \mathbb{E} \left[ g(X_T \lor \bar X_T, {\bar{\mu}}_T)- g(\bar X_T, {\bar{\mu}}_T) \right].
\end{align}
Next, from Condition \ref{assumption submodularity multidimensional for J} in Assumption \ref{assumption for J} we have 
\begin{align*}
\mathbb{E} \bigg[ \int_0^T \Big( f(t,X_t \lor \bar X_t, {\bar{\mu}}_t)  - f(t,\bar X_t, {\bar{\mu}}_t) \Big) dt  \bigg]  
&\leq \mathbb{E} \bigg[ \int_0^T  \Big( f(t,X_t , {\bar{\mu}}_t)  -f(t,X_t \land \bar X_t, {\bar{\mu}}_t) \Big) dt  \bigg]
 \\
&\leq \mathbb{E} \bigg[ \int_0^T  \Big( f(t,X_t , {{\mu}}_t)  -f(t,X_t \land \bar X_t, {{\mu}}_t) \Big) dt  \bigg],   
\end{align*}
as well as
\begin{align*} 
\mathbb{E} \left[ g(X_T \lor \bar X_T, {\bar{\mu}}_T)  - g(\bar X_T, {\bar{\mu}}_T) \right] 
& \leq \mathbb{E} \left[ g(X_T , {\bar{\mu}}_T)  -g(X_T \land \bar X_T, {\bar{\mu}}_T) \right] \\
& \leq \mathbb{E} \left[ g(X_T , {{\mu}}_T)  -g(X_T \land \bar X_T, {{\mu}}_T) \right].
\end{align*}
Moreover, by the definition of $\alpha^{\vee}$ and $\alpha^\wedge$, we see that
\begin{align*}
\mathbb{E} \bigg[ \int_0^T \Big( l^i(t,X_t^i \lor \bar X_t^{i}, \alpha_t^{i,\vee}) &-l^i(t,\bar X_t^{i}, \bar \alpha _t^{i}) \Big) dt \bigg] \\
& =  \mathbb{E} \bigg[ \int_0^T  \mathds{1}_{\{ X_t^i \geq  \bar X _t^i \}}  \Big( l^i(t,X_t^i, {\alpha}_t^{i}) - l^i(t, \bar X_t^{i}, \alpha_t^{i,\wedge}) \Big) dt \bigg] \\
& = \mathbb{E} \bigg[ \int_0^T \Big( l^i(t,X_t^i, \alpha_t^i ) -l^i(t,X_t^i \land \bar X_t^i, \alpha_t^{i,\wedge}) \Big) dt \bigg]. 
\end{align*} 
Finally, the latter three inequalities allow to estimate (\ref{J-J}), thus obtaining
\begin{align*}
0 \leq J(\alpha^{\vee},\bar{\mu}) - J(\bar \alpha,{\bar{\mu}})  
& \leq    \mathbb{E} \bigg[ \int_0^T \Big(f(t,X_t , {{\mu}}_t)  -f(t,X_t \land \bar X_t, {{\mu}}_t) \Big) dt  \bigg] \\ \notag
& \quad \quad + \mathbb{E} \bigg[ \int_0^T \sum_{i=1}^d \Big( l^i(t,X_t^i, \alpha_t^i ) -l^i(t,X_t^i \land   \bar X_t^i, \alpha_t^{i,\wedge}) \Big) dt \bigg] \\ \notag
& \quad \quad \quad \quad + \mathbb{E} \left[ g(X_T , {{\mu}}_T)- g(X_T \land \bar  X_T, {{\mu}}_T) \right],
\end{align*}
which, thanks to Lemma \ref{lemma lattice of trajectories}, can be written as
\begin{equation}
\label{J_J.switch}
    0\leq J(\alpha^{\vee},\bar{\mu}) - J(\bar \alpha ,{\bar{\mu}}) \leq   J(\alpha ,{\mu}) -  J(\alpha^\wedge,{\mu}). 
\end{equation}

Hence, by optimality of $\alpha$, the control $\alpha^\wedge $ is also a minimizer for $J(\cdot,\mu)$, and, by uniqueness of optimal controls (see Lemma \ref{lemma existence of optimal controls}), we conclude that 
$\alpha=\alpha^\land$. Therefore, by Lemma \ref{lemma lattice of trajectories} we obtain $X \land  \bar X=X$; that is, $X \leq \bar X $, completing the proof. 

\subsection{Proof of Proposition \ref{proposition best reply map increasing} under Assumption \ref{assumption for comparison}}
    \label{subsection proof best reply increasing: comparison}
Again, by monotonicity of the map $p$ (see  \eqref{eq projection monotone}), it is sufficient to prove only Claim \ref{proposition best reply map increasing: claim Gamma}.

The proof is divided in two steps. 
\smallbreak\noindent
\emph{Step 1.}
In this step we prove the monotonicity of the feedback function $\hat \alpha$ providing the optimal controls (see Lemma \ref{lemma existence of optimal controls}).  
In particular, we show that, for any $(t,x,\mu,y), \, (t,\bar x, \bar \mu, \bar y) \in [0,T] \times \R^d \times \mathcal P _2 (\R^d) \times \R^d$, 
\begin{equation}
    \label{eq comparative alpha}
   \text{ if $ x \leq \bar x$, $\bar y \leq y$, $\mu \leq^{\text{\tiny{\rm st}}}  \bar \mu$, then } \hat \alpha (t,x,\mu,y) \leq \hat \alpha (t,\bar x, \bar \mu, \bar y). 
\end{equation}

Fix
$(t,x,\mu,y), \, (t,\bar x, \bar \mu, \bar y) \in [0,T] \times \R^d \times \mathcal P _2 (\R^d) \times \R^d$
with 
$ x \leq \bar x$, $\bar y \leq y$, $\mu \leq^{\text{\tiny{\rm st}}}  \bar \mu$. 
Set $
\alpha := \hat \alpha (t,x,\mu,y)$ and  $\bar \alpha :=  \hat \alpha (t,\bar x, \bar \mu, \bar y)$.
Since $\bar \alpha$ minimizes the function $a\mapsto y b_2(t)a + h(t,\bar x, \bar \mu, a)$, we have
\begin{align*}
    0 \leq \bar y b_2(t) \alpha \lor \bar \alpha + h(t,\bar x, \bar \mu, \alpha \lor \bar \alpha) - \bar y b_2(t)\bar \alpha + h(t,\bar x, \bar \mu, \bar \alpha),
\end{align*} 
so that, by using Condition \ref{assumption submodularity multidimensional with comparison.h submodular a} in Assumption \ref{assumption for comparison} and the linearity of the map $a\mapsto y b_2(t)a$, we obtain 
\begin{align*}
    0 & \leq  \bar y b_2(t) \alpha \lor \bar \alpha + h(t,\bar x, \bar \mu, \alpha \lor \bar \alpha) - \bar y b_2(t)\bar \alpha + h(t,\bar x, \bar \mu, \bar \alpha) \\
    & \leq \bar y b_2(t) \alpha + h(t,\bar x, \bar \mu, \alpha ) - \bar y b_2(t) \alpha \land \bar \alpha + h(t,\bar x, \bar \mu,  \alpha \land \bar \alpha).
\end{align*}
From the latter inequality, we can use Condition \ref{assumption submodularity multidimensional with comparison.h decreasing differences} in Assumption \ref{assumption for comparison}, together with the fact that the map $y \mapsto y b_2(t) (\alpha - \alpha \land \bar \alpha)$ is nonincreasing (by Condition \ref{assumption monotonicity multidimensional with comparison} in Assumption \ref{assumption for comparison}), in order to obtain 
\begin{align*}
    0  & \leq \bar y b_2(t) \alpha + h(t,\bar x, \bar \mu, \alpha ) - \bar y b_2(t) \alpha \land \bar \alpha + h(t,\bar x, \bar \mu,  \alpha \land \bar \alpha) \\ 
    & \leq y b_2(t) \alpha + h(t, x,  \mu, \alpha ) - y b_2(t) \alpha \land \bar \alpha + h(t, x,  \mu,  \alpha \land \bar \alpha).
\end{align*}
This implies that $\alpha \land \bar \alpha$ minimizes the function $ a \mapsto y b_2(t)a + h(t, x, \mu, a)$. Therefore, since the minimizer $\alpha$ is unique by Remark \ref{remark Lipschitz feedback}, we obtain that $\alpha \land \bar \alpha= \alpha$, so that \eqref{eq comparative alpha} is proved. 

\smallbreak\noindent
\emph{Step 2.}
Take $\mu, \, \bar \mu \in  \mathcal M _{\text{\tiny{$B$}}}^2$ with $\mu \leq^{\text{\tiny{\rm st}}}  \bar \mu$ and let $(\bar X, \bar \alpha )$ and $(X,\alpha)$ be the optimal pairs related to $\bar{\mu} $ and $\mu$, respectively. 
By Lemma \ref{lemma existence of optimal controls}, $X$ and $\bar X$ are given by the unique solutions $(X,Y,Z,Z^\circ)$ and $(\bar X, \bar Y,\bar Z,\bar Z ^\circ)$ to the FBSDEs depending on $\mu$ and $\bar \mu$, respectively. 

We now aim at using the monotonicity of $\hat \alpha$ proved in Step 1 (see \eqref{eq comparative alpha}) and Assumption \ref{assumption for comparison} to employ a comparison theorem for FBSDEs (see Theorem 2.2 in \cite{chen.luo2021}) in order to show that 
\begin{equation}
    \label{eq comparison pronciple forward component}
X \leq \bar X.
\end{equation}
First, observe that Assumptions $\mathscr A 1 - \mathscr A 4$ in \cite{chen.luo2021} (with the extra growth in $\mu$ and $\bar \mu$) clearly follow from our Assumptions \ref{assumption existence} and \ref{assumption for comparison}, together with the Lipschitz continuity and the growth of $\hat \alpha $ (see Remark \ref{remark Lipschitz feedback}).
In particular, we underline that, (with minimal adjustment) the results in \cite{chen.luo2021} cover also our case, in which the growth conditions in Assumptions $\mathscr A 1 - \mathscr A 4$ in \cite{chen.luo2021} are replaced with
$$
|\hat b (t,x,\nu,y)| \leq C(1 +|x| + \| \nu \|_1 + |y|) , \quad  | \hat h (t,x,\nu,y)| + |\hat g (x,\nu)| \leq C(1 + \| \nu \|_1 +|y|), 
$$
for any $(t,x,\nu,y) \in [0,T] \times \R^d \times \mathcal P _2 (\R^d) \times \R^d$.
We remain to check the conditions  in \cite{chen.luo2021} on the monotonicity of $(\hat b, \hat h, \hat g)$. 
To do so, take $(t,x,\nu,y), \, (t,\bar x, \bar \nu, \bar y) \in [0,T] \times \R^d \times \mathcal P _2 (\R^d) \times \R^d$ with  $ x \leq \bar x$, $\bar y \leq y$, $\nu \leq^{\text{\tiny{\rm st}}}  \bar \nu$.  
 From the monotonicity $\hat \alpha$ (see \eqref{eq comparative alpha}) and from Condition \ref{assumption monotonicity multidimensional with comparison} in Assumption \ref{assumption for comparison} we obtain:
\begin{equation}
    \label{eq monotonicity b}
    \hat b (t,x,\nu,y) \leq \hat b (t,\bar x,\bar \nu, \bar y).
\end{equation}
 From Conditions \ref{assumption monotonicity multidimensional with comparison} and \ref{assumption submodularity multidimensional with comparison.h for using comparison principle} in Assumption \ref{assumption for comparison}, and from the monotonicity $\hat \alpha$  we have: 
\begin{equation}
    \label{eq monotonicity of h and g}
    \hat h (t,\bar x,\bar \nu, \bar y) \leq  \hat h (t, x, \nu,  y) \quad \text{and} \quad  \hat g (\bar x,\bar \nu) \leq  \hat g ( x, \nu).
\end{equation}
The conditions \eqref{eq monotonicity b} and \eqref{eq monotonicity of h and g} are a slightly different version of Assumptions $\mathscr A 5 - \mathscr A 7$ and of the conditions in the statement of Theorem 2.2 in \cite{chen.luo2021}. 
In particular, the inequalities in \eqref{eq monotonicity of h and g} have opposite direction with respect to the conditions in \cite{chen.luo2021}, and the monotonicity in $y$ in \eqref{eq monotonicity b} is inverted as well.  
This implies that we can employ a slightly different version of Theorem 2.2 in \cite{chen.luo2021}, which gives $X \leq \bar X$ and $\bar Y \leq Y$. 
This completes the proof of the proposition.

\subsection{Proof of Lemma \ref{lemma a priori estimates}}
    \label{subsection proof a priori estimate}
We begin by proving the uniform a priori estimate \eqref{eq moment a priori estimates} on the moments of the optimally controlled state processes. 
For $\alpha \in \mathcal A$ and $\mu \in  \mathcal M _{\text{\tiny{$B$}}}^2$, let $X^{\alpha,\mu}$ denotes the solution to the SDE$_{(\alpha,\mu)}$  \eqref{eq controlled SDE}.  
By boundedness of $b_1$ in $\mu$, classical estimates give
\begin{equation}
    \label{eq estimate SDE with control}
    \mathbb E \bigg[ \sup_{t\in[0,T]} |X_t^{\alpha,\mu}|^2 \bigg] 
    \leq C \bigg( 1 + \mathbb E [|\xi|^2] + \mathbb E \bigg[ \int_0^T |\alpha_t|^2 dt \bigg] \bigg).
\end{equation} 
If now $A$ is compact, the estimate \eqref{eq moment a priori estimates} is obvious.
Suppose instead that Condition \ref{assumption a priori estimates coercivity} in Assumption \ref{assumption a priori estimates} holds. 
Denote by $\alpha^\mu$ the optimal control for $\mu$ (by Lemma \ref{lemma existence of optimal controls} such a control exists unique) and let $\bar{\alpha} \in \mathcal A$ be the control constantly equal to $a_0$ (see Assumption \ref{assumption existence}).
By using Condition \ref{assumption existence: growth conditions on costs} in Assumption \ref{assumption existence} and the optimality of $\alpha^\mu$, we obtain
$$
\kappa \bigg( \mathbb E \bigg[ \int_0^T |\alpha_t^\mu|^2 dt \bigg] -T -1 \bigg) \leq J(\alpha^\mu,\mu) \leq J(\bar \alpha,\mu) 
\leq C \bigg(  1+ \mathbb E \bigg[  \sup_{t\in[0,T]} |X_t^{\bar{\alpha},\mu}|^2 + \int_0^T | \bar \alpha _t|^2 dt \bigg] \bigg),
$$
so that, by \eqref{eq estimate SDE with control}, we deduce that
$$ 
\mathbb E \bigg[ \int_0^T |\alpha_t^\mu|^2 dt \bigg] \leq C \bigg( 1+ \mathbb E \bigg[ \int_0^T  | \bar \alpha _t|^2 dt \bigg] \bigg) = C(1+a_0 T), 
$$
for a constant $C$ which does not depend on $\mu$. 
From the latter estimate, we can again employ  \eqref{eq estimate SDE with control} in order to conclude that 
$$
\mathbb E \bigg[ \sup_{t\in[0,T]} |X_t^{\mu}|^2  \bigg] \leq C,
$$
which (since the generic constant $C$ does not depend on $\mu$) implies the uniform estimate \eqref{eq moment a priori estimates}. 

We next prove that $\overline{\Gamma} \in \bH^2$.
In doing so, the monotonicity of the best reply map $R$ (see Proposition \ref{proposition best reply map increasing}) will play an essential role.
First, observe that there exists a sequence $(N^n)_n \subset \bH^2$ such that
$$
\overline \Gamma = \sup_n R(N^n), \quad \mathbb P \otimes \pi \text{-a.e.\ in $\Omega \times [0,T]$}.
$$
Recalling the notation introduced in \eqref{eq definition land lor processes}, construct the sequences  
$$
N^{\land,n}:= N^n \land...\land N^0 \quad \text{and} \quad N^{\lor,n}:= N^n \lor...\lor N^0, \quad n\in \mathbb N.
$$
Clearly, we have
$$
\mathbb E \bigg[ \int_0^T (|N^{\land,n}_t|^2 + |N^{\lor,n}_t|^2) d\pi(t) \bigg] \leq 2 \sum_{j=1}^n \mathbb E \bigg[ \int_0^T |N^j_t|^2 d\pi(t) \bigg],  
$$
so that the processes $N^{\land,n}$ and $N^{\lor,n}$ are square integrable; that is,
\begin{equation}
    \label{eq N sup inf square integrable} 
    N^{\land,n}, \, N^{\lor,n} \in \mathbb H ^2.
\end{equation}

Observe next that, since $N^{\land,n} \leq N^n \leq N^{\lor,n}$,  by Proposition \ref{proposition best reply map increasing} we have
$R(N^{\land,n}) \leq R( N^n) \leq R(N^{\lor,n})$, so that
\begin{equation}\label{eq comparison with monotone sequences}
R(N^{\land,n})\land 0 \leq R( N^n)\land 0  \leq R( N^n)\lor 0 \leq R(N^{\lor,n}) \lor 0.
\end{equation}
Also, since the sequences $(N^{\land,n})_n$ and $(N^{\lor,n})_n$ are monotone, again by Proposition \ref{proposition best reply map increasing} we have that $(R(N^{\land,n}))_n$ is decreasing and that $(R(N^{\lor,n}))_n$ is increasing. 
Therefore, the sequences
\begin{equation}\label{eq contruction of monotone responses}
(|R(N^{\land,n})\land 0|^2 )_n \text{ and } (|R(N^{\lor,n}) \lor 0|^2)_n \text{ are increasing}.
\end{equation}
 
We can now estimate $|\overline \Gamma|^2$. 
Using \eqref{eq comparison with monotone sequences} we find 
\begin{align*} 
|\overline \Gamma|^2 & = \Big|\sup_n R(N^n)\Big |^2 
= \Big|\sup_n (R(N^n)\land 0 +R(N^n)\lor 0)   \Big |^2 \\
& \leq \sup_n |R(N^n)\land 0|^2 + \sup_n |R(N^n)\lor 0|^2 \\
& \leq \sup_n |R(N^{\land,n})\land 0|^2 + \sup_n |R(N^{\lor,n})\lor 0|^2.
\end{align*}
Hence, by the monotonicity in \eqref{eq contruction of monotone responses}, we obtain
$$
|\overline \Gamma|^2  \leq \lim_n \Big( |R(N^{\land,n})\land 0|^2 + |R(N^{\lor,n})\lor 0|^2 \Big).
$$
The latter allows to use Fatou's lemma in order to find
\begin{align*}
\E \bigg[ \int_0^T |\overline \Gamma_t |^2 d\pi(t) \bigg]  & \leq  \E \bigg[ \int_0^T \lim_n \Big( |R(N^{\land,n})_t\land 0|^2 + |R(N^{\lor,n})_t\lor 0|^2 \Big) d\pi(t) \bigg]\\
& \leq \liminf_n \E \bigg[ \int_0^T \Big( |R(N^{\land,n})_t\land 0|^2 + |R(N^{\lor,n})_t\lor 0|^2 \Big)  d\pi(t) \bigg]\\
& \leq \liminf_n \E \bigg[ \int_0^T \Big( |R(N^{\land,n})_t|^2 + |R(N^{\lor,n})_t|^2 \Big) d\pi(t) \bigg] < \infty, 
\end{align*}
where the last inequality follows from \eqref{eq moment a priori estimates} and from \eqref{eq N sup inf square integrable}. 
Hence we have $\overline{\Gamma} \in \bH^2$.

In the same way, it is possible to show that $\underline{\Gamma} \in \bH^2$, completing the proof of the lemma.

\section{Proof of Theorems \ref{theorem main FBSDE}, \ref{theorem main} and \ref{theorem main fictitious play}}
 \label{section proof main results}
\subsection{Proof of Theorem \ref{theorem main FBSDE}}
     \label{subsection proof first main result}
We prove the two claims separately.
\subsubsection{Proof of Claim \ref{theorem.main.existence.FBSDE}: the fixed point argument}
Define the set 
$$
 L := \{ M \in \bH^2 | \underline \Gamma \leq M \leq \overline \Gamma \}.
$$
The partially ordered set  $(L, \leq)$ is a complete lattice (see \eqref{eq definition land lor processes}), and, by the definition of $\underline \Gamma$ and $\overline \Gamma$ (see \eqref{eq definition Gamma upper lower bounds}), we have $R(M) \in  L$ for any $M\in \bH^2$.
Moreover, by Proposition \ref{proposition best reply map increasing}, the best reply map $R$ is increasing from $ L$ into itself.   
We can therefore employ Tarski's fixed point theorem (see \cite{T}) in order to deduce that the set of fixed points of the map ${R}: {L}\to {L} $ is a nonempty complete lattice. 
In light of Remark \ref{remark characterization of MFGE}, the set of fixed points of the map $R$ coincides with the set of (forward components of the) solutions to the MKV FBSDE \eqref{eq MKV FBSDE}, completing the proof of Claim \ref{theorem.main.existence.FBSDE}.

\subsubsection{Proof of Claim \ref{theorem.main.general.convergence.FBSDE}: convergence of the algorithm}
    \label{Proof of Claim theorem.main.general.convergence.FBSDE}
It is sufficient to prove  Claim \ref{theorem.main.general.convergence.up.FBSDE}, as the proof of Claim \ref{theorem.main.generalconvergence.down.FBSDE} follows by similar arguments. 
The proof is divided in four steps. 
\smallbreak\noindent
\emph{Step 1.}
We begin by observing that, by the definition of $\underline \Gamma$ (see \eqref{eq definition Gamma upper lower bounds}) and by the monotonicity of $R$ (see Proposition \ref{proposition best reply map increasing}), one has 
$$
\underline{X}^0 = \underline \Gamma \leq R (\underline \Gamma) = \underline X ^1, 
$$
so that, iterating the map $R$, the monotonicity of $R$ gives
\begin{equation}\label{eq X n increasing}
\underline{X}^n \leq \underline X ^{n+1}, \quad  n \in \mathbb N.
\end{equation}
Therefore, one can define the process 
\begin{equation}\label{eq point conv of X n to X}
\underline X  := \sup_n \underline X ^n= \lim_n \underline X ^n, \quad \mathbb P \otimes \pi \text{-a.e.\ in $\Omega \times [0,T]$}. 
\end{equation}
Obviously, since $\underline \Gamma \leq \underline X ^n \leq \overline \Gamma$ for every $n \in \mathbb N$, we have
\begin{equation}
    \label{eq a priori of everything} 
\text{$|\underline X _t |^2, \, |\underline X _t^n|^2 \leq |\underline \Gamma _t|^2 + |\overline \Gamma _t|^2 $ for any $n \in \mathbb N$}, \quad \mathbb P \otimes \pi \text{-a.e.\ in $\Omega \times [0,T]$}, 
\end{equation}
 so that, by integrability of $\underline \Gamma$ and of $\overline \Gamma$ (see Lemma \ref{lemma a priori estimates}) and by the dominated convergence theorem one obtains 
\begin{equation}
    \label{eq convergence X n to X sup in L2}
    \lim_n \E \bigg[ \int_0^T |\underline X _t^n - \underline X _t|^2 d\pi(t) \bigg] =0.
\end{equation}
Moreover, consider the stochastic flows $ \underline{ \mu}^n,\, \underline \mu \in  \mathcal M _{\text{\tiny{$B$}}}^2$  given by 
\begin{equation}
    \label{eq definition mu n and mu inside the proof}
    \underline{ \mu}^n =( \mathbb P _{\text{\tiny{$\underline X ^{n}_t|B$}}})_{t\in [0,T]}, \quad n\in \mathbb N, \quad \underline{ \mu} :=( \mathbb P _{\text{\tiny{$\underline X _t|B$}}})_{t\in [0,T]}.
\end{equation} 
Notice that the processes $\underline{ \mu}^n$ are those defined in  \eqref{eq learning procedure mu}.
From \eqref{eq X n increasing} and \eqref{eq point conv of X n to X}, we have that the sequence $| \underline X _t^n  - \underline X _t |^2$ is decreasing and it converges to 0.
Moreover, by \eqref{eq a priori of everything} and by Lemma \ref{lemma a priori estimates}, we have $\mathbb E [| \underline X _t^0  - \underline X _t |^2] \leq 4 \mathbb E [ |\underline \Gamma _t|^2+  |\overline \Gamma _t|^2|] < \infty$ for $\pi$-a.a $t \in [0,T]$. 
Thus, by the monotone convergence theorem for the conditional expectation, together with elementary properties of the 2-Wasserstein distance $\ell_{\R^d,2}$ on $\R^d$ (see \eqref{eq Wasserstain distance}), we obtain
$$
\lim_n \ell_{\R ^d, 2}^2 ( \underline \mu_t ^n , \underline \mu_t )  \leq \lim_n \mathbb E \big[ | \underline X _t^n  - \underline X _t |^2 |  B \big] =0, \quad \text{$\mathbb P$-a.s., for $\pi$-a.a.\ $t\in [0,T]$,}
$$
Furthermore, by \eqref{eq X n increasing} and the monotonicity of the projection map $p$ (see \eqref{eq projection monotone}), the latter limit is monotonic w.r.t.\ $\leq^{\text{\tiny{$\rm st$}}}$; that is, 
\begin{equation}\label{eq point conv of mu n to mu}
    \underline{ \mu}^n \leq^{\text{\tiny{$\rm st$}}} \underline{ \mu}^{n+1} \leq^{\text{\tiny{$\rm st$}}} \underline{ \mu}, \quad \text{and} \quad \ell_{\R^d, 2} (\underline{ \mu}_t^n, \underline{ \mu}_t) \to 0 \text{ as $n\to \infty$, $\mathbb P$-a.s., for $\pi$-a.a.\ $t\in [0,T]$.}
\end{equation} 

By the definition of  $(\underline X ^{n}, \underline Y ^{n}, \underline Z ^{n}, \underline Z ^{\circ, n})$ and $\underline \mu^n$, the process $(\underline X ^{n}, \underline Y ^{n}, \underline Z ^{n}, \underline Z ^{\circ, n})$ solves the system 
\begin{align}
    \label{eq FBSDE n}
    & d \underline X _t ^n = \hat b ( t, \underline X _t^n, \underline \mu _t^{n-1}, \underline Y _t^n ) dt + \sigma (t, \underline X_t^n) d W_t + \sigma^\circ (t, \underline X_t^n) d B_t, \quad \underline X _0^n = \xi, \\  \notag
    & d \underline Y _t^n = - \hat h (t,\underline X _t^n,\underline \mu _t ^{n-1}, \underline Y _t^n) dt + \underline Z _t^n dW_t + \underline Z _t^{\circ,n} dB_t, \quad \underline Y _T^n = \hat g( \underline X _T^n, \underline \mu _T^{n-1}). \notag
\end{align}
for any $n \in \mathbb{N}$ with $n \geq 1$. 
The idea of the proof is now to take (monotonic) limits in the  system above in order to approximate the minimal solution of the MKV FBSDE \eqref{eq MKV FBSDE}. 
This will be done in the subsequent two steps.
\smallbreak\noindent
\emph{Step 2.} 
In this step we will prove that the process $(\underline Y^n, \underline Z ^n, \underline Z ^{\circ,n})$ converges to the unique strong solution $(\underline Y, \underline Z, \underline Z ^\circ)$ of the backward stochastic differential equation (BSDE, in short)
\begin{equation}\label{eq limi BSDE with X}
d \underline Y _t = - \hat h (t,\underline X _t, \underline \mu _t, \underline Y _t) dt + \underline Z _t dW_t + \underline Z _t^{\circ} dB_t, \quad \underline Y _T = \hat g( \underline X _T, \underline \mu _T). 
\end{equation} 
Recall that a solution of such a BSDE is defined as a process $(\underline Y, \underline Z, \underline Z ^\circ)$  such that
\begin{equation}
    \label{eq BSDE integrability}
    \E \bigg[ \sup_{ t \in [0,T]} | \underline Y _t|^2 
    + \int_0^T ( |  \underline Z _t |^2+ |  \underline Z _t^\circ |^2 )dt \bigg]  < \infty,
\end{equation}
and such that the equation \eqref{eq limi BSDE with X} is satisfied.

By \eqref{eq FBSDE n}, the processes $(\underline Y^n, \underline Z ^n, \underline Z ^{\circ,n})$ solve the BSDEs
$$
d \underline Y _t^n = - \hat h (t,\underline X _t^n,\underline \mu _t ^{n-1}, \underline Y _t^n) dt + \underline Z _t^n dW_t + \underline Z _t^{\circ,n} dB_t, \quad \underline Y _T^n = \hat g( \underline X _T^n, \underline \mu _T^{n-1}).
$$
Therefore, by stability properties for BSDEs (see, e.g., Theorem 4.2.3 at p.\ 84 in \cite{zhang2017}) we obtain a first estimate
\begin{align}
    \label{eq estimate stability BSDE}
& \E \bigg[ \sup_{ t \in [0,T]} |\underline{Y}_t^n - \underline Y _t|^2 
    + \int_0^T ( | \underline{Z}_t^n - \underline Z _t |^2+ | \underline{Z} _t^{\circ,n} - \underline Z _t^\circ |^2 )dt \bigg] \\ \notag 
& \leq C \E \bigg[ |\hat g (\underline X _T^n, \underline \mu _T^{n-1})- \hat g (\underline X_T, \underline \mu _T) |^2   + \int_0^T | \hat h (t, \underline X_t^n, \underline \mu _t^{n-1}, \underline Y _t) - \hat h (t,\underline X _t, \underline \mu _t, \underline Y _t) |^2 dt \bigg ],
\end{align} 
for a suitable constant $C$ which does not depend on $n$.
Hence, we proceed by estimating the right hand side in \eqref{eq estimate stability BSDE}.

We begin by studying the convergence of the terms with $\hat g$. 
Write
\begin{align}
    \label{eq estimate stability terminal condition}  
 \E \Big[ |\hat g (\underline X _T^n, \underline \mu _T^{n-1})- \hat g (\underline X_T, \underline \mu _T) |^2 \Big] \leq &
 2 \E \Big[ | D_x g (\underline X _T^n, \underline \mu _T^{n-1})-  D_x g (\underline X_T, \underline \mu _T^{n-1}) |^2 \Big] \\ \notag   
 & +2  \E \Big[ | D_x g (\underline X_T, \underline \mu _T^{n-1})-  D_x g (\underline X_T, \underline \mu _T) |^2 \Big].
\end{align}
From the monotonicity of the sequence $(\underline \mu ^n)_n$ in \eqref{eq point conv of mu n to mu} and the submodularity assumptions
(in particular, either Condition \ref{assumption submodularity multidimensional for J} in Assumption \ref{assumption for J} or  Condition \ref{assumption submodularity multidimensional with comparison} in Assumption \ref{assumption for comparison}), we have 
\begin{align*}
 D_x g (\underline X_T, \underline \mu _T^{n})-  D_x g (\underline X_T, \underline \mu _T) \geq D_x g (\underline X _T, \underline \mu _T^{n+1})- D_x g (\underline X _T, \underline \mu _T) \geq  0,
\end{align*} 
so that 
$$
 |D_x g (\underline X _T, \underline \mu _T^n)- D_x g (\underline X _T, \underline \mu _T)|^2 \geq | D_x g (\underline X _T, \underline \mu _T^{n+1})- D_x g (\underline X _T, \underline \mu _T) |^2. 
$$  
Also, by the growth condition on $D_xg$ (from Condition \ref{assumption existence: growth conditions on costs} in Assumption \ref{assumption existence}), from \eqref{eq a priori of everything} we obtain
\begin{align*}
 \mathbb E [|D_x g (\underline X _T, \underline \mu _T^0)- D_x g (\underline X _T, \underline \mu _T)|^2] & \leq C \mathbb E [ 1 +|\underline X _T|^2  + \| \underline \mu _T^0 \|_1^2 + \| \underline \mu _T \|_1^2]  \\
 & \leq C \mathbb E \big[1+|\underline \Gamma _T|^2 + |\overline \Gamma _T|^2  + \|\mathbb P _{\text{\tiny{$\underline \Gamma _T|B$}}} \|_1^2 + \| \mathbb P _{\text{\tiny{$\underline X _T|B$}}} \|_1^2 \big] \\
 & \leq C \mathbb E [1+|\underline \Gamma _T|^2 + |\overline \Gamma _T|^2 ] < \infty,
\end{align*}
where the latter estimates follows from Lemma \ref{lemma a priori estimates}.
Therefore, by the limits in \eqref{eq point conv of mu n to mu} and the continuity of $D_x g$ (from Assumption \ref{assumption continuity in the measure}), we can use the monotone convergence theorem and obtain
\begin{equation}
    \label{eq convergence Dg 1}
    \lim_n \E \Big[ |D_x g (\underline X _T, \underline \mu _T^n)- D_x g (\underline X _T, \underline \mu _T)|^2 \Big] =0.
\end{equation}  
Finally, by using the uniform Lipschitz continuity of $D_x g$ (from Condition \ref{assumption existence: growth conditions on costs} in Assumption \ref{assumption existence}) and the limits in \eqref{eq convergence X n to X sup in L2}, we obtain 
\begin{equation} 
    \label{eq convergence Dg 2}
    \lim_n \E \Big[ | D_x g (\underline X _T^n, \underline \mu _T^{n-1})-  D_x g (\underline X_T, \underline \mu _T^{n-1}) |^2 \Big]  \leq  C \lim_n \mathbb{E} [|\underline X_T^n - \underline X_T|^2 ] =0,  
\end{equation} 
so that, combining \eqref{eq convergence Dg 1} and \eqref{eq convergence Dg 2} into \eqref{eq estimate stability terminal condition}, we conclude that 
\begin{equation}
    \label{eq convergence Dg}
    \lim_n \E \Big[ |\hat g (\underline X _T^n, \underline \mu _T^{n-1})- \hat g (\underline X_T, \underline \mu _T) |^2 \Big] =0. 
\end{equation}

We next study the convergence of the terms with $\hat h$ in \eqref{eq estimate stability BSDE}.
We write
\begin{align}
    \label{eq estimate stability integral}  
 \E \bigg[ \int_0^T | \hat h (t, \underline X_t^n, \underline \mu _t^{n-1}, \underline Y _t) & - \hat h (t,\underline X _t, \underline \mu _t, \underline Y _t) |^2 dt \bigg ] \\ \notag
& \leq 2
 \E \bigg[ \int_0^T | \hat h (t, \underline X_t^n, \underline \mu _t^{n-1}, \underline Y _t) - \hat h (t,\underline X _t, \underline \mu _t^{n-1}, \underline Y _t) |^2 dt \bigg ]  \\  \notag
 &\quad  + 2 \E \bigg[ \int_0^T | \hat h (t, \underline X_t, \underline \mu _t^{n-1}, \underline Y _t) - \hat h (t,\underline X _t, \underline \mu _t, \underline Y _t) |^2 dt \bigg ] \\ \notag
 & =: 2 \E \bigg[ \int_0^T (|\hat H _t^{1,n}|^2 +|\hat H _t^{2,n}|^2) dt \bigg], 
 \end{align} 
 and we  proceed by estimating separately the two terms in the right hand side of \eqref{eq estimate stability integral}.

We begin with the term involving $\hat H ^{1,n}$. 
By the Lipschitz continuity of $D_x h$ in $(x,a)$ (from Condition \ref{assumption existence: growth conditions on costs} in Assumption \ref{assumption existence}) we obtain 
 \begin{align*}
|\hat H _t^{1,n}|^2 
 \leq & C \Big(   |  \underline X_t^n - \underline X _t|^2 + | \hat \alpha (t, \underline X_t^n, \underline \mu _t^{n-1}, \underline Y _t) - \hat \alpha (t,\underline X _t, \underline \mu _t^{n-1}, \underline Y _t) |^2   \\  \notag 
 & + | (D_x b_1 (t, \underline X_t^n, \underline \mu _t^{n-1}) - D_x b_1 (t,\underline X _t, \underline \mu _t^{n-1})) \underline Y _t |^2 \Big),   
 \end{align*}
 so that, using the Lipschitz continuity of $\hat \alpha$ in $x$ (see Remark \ref{remark Lipschitz feedback}) we find
 \begin{align}\label{eq estimate integral 1}
|\hat H _t^{1,n}|^2 
& \leq   C \Big( |  \underline X_t^n - \underline X _t|^2  + | (D_x b_1 (t, \underline X_t^n, \underline \mu _t^{n-1}) - D_x b_1 (t,\underline X _t, \underline \mu _t^{n-1})) \underline Y _t |^2 \Big).
 \end{align}
Moreover, by continuity of $D_x b_1$ in $x$ (indeed, under either Assumption \ref{assumption for J} or \ref{assumption for comparison}, $D_x b_1$ does not depend on the measure) and the limits in \eqref{eq point conv of X n to X},
we get  
$$
| (D_x b_1 (t, \underline X_t^n, \underline \mu _t^{n-1}) - D_x b_1 (t,\underline X _t, \underline \mu _t^{n-1})) \underline Y _t |^2 \to 0, \text{ as $n \to \infty$},
$$ 
and, by boundedness of $D_x b_1$ we have 
$$
| (D_x b_1 (t, \underline X_t^n, \underline \mu _t^{n-1}) - D_x b_1 (t,\underline X _t, \underline \mu _t^{n-1})) \underline Y _t |^2  \leq C |\underline Y_t|^2.
$$
Since the left hand side in the latter inequality is integrable (by \eqref{eq BSDE integrability}), by the dominated convergence theorem, we obtain 
$$ 
\lim_n \E \bigg[ \int_0^T | (D_x b_1 (t, \underline X_t^n, \underline \mu _t^{n-1}) - D_x b_1 (t,\underline X _t, \underline \mu _t^{n-1})) \underline Y _t |^2  dt \bigg] =0.
$$ 
The latter equality, together with \eqref{eq estimate integral 1} and \eqref{eq convergence X n to X sup in L2},  allows to conclude that 
\begin{equation}
    \label{eq convergence integral 1}
    \lim_n \E \bigg[ \int_0^T |\hat H _t^{1,n}|^2  dt \bigg] = 0. 
\end{equation}

We continue by estimating the term in \eqref{eq estimate stability integral} involving $\hat H ^{2,n}$.  
Observe that, under both Assumptions \ref{assumption for J} or \ref{assumption for comparison}, we have $ D_x b_1 (t, \underline X_t, \underline \mu _t^{n-1}) = D_x b_1 (t, \underline X_t, \underline \mu _t)$, so that 
\begin{align} \label{eq estimate integral 2}
  \hat H_t ^{2,n} =   D_x h (t, \underline X_t, \underline \mu _t^{n-1},  \hat \alpha( t, \underline X_t, \underline \mu _t^{n-1}, \underline Y _t)) - D_x h (t,\underline X _t, \underline \mu _t, \hat \alpha( t,\underline X _t, \underline \mu _t, \underline Y _t)).
\end{align} 
Now, under Assumption \ref{assumption for comparison}, we have shown in the proof of Proposition \ref{proposition best reply map increasing} (see \eqref{eq comparative alpha}) that $\hat \alpha$ is monotone in the measure $\nu$. 
On the other hand, under Assumption \ref{assumption for J}, $\hat \alpha$ does not depend on the measure. 
In both cases, from the monotonicity of the sequence $(\underline \mu ^n)_n$ in \eqref{eq point conv of mu n to mu} and the submodularity  assumptions (i.e., either Assumption \ref{assumption for J} or \ref{assumption for comparison}), we have 
\begin{align*}
 & D_x h (t, \underline X _t, \underline \mu _t^n, \hat \alpha (t, \underline X _t, \underline \mu _t^n, \underline Y _t ))- D_x h (t, \underline X _t, \underline \mu _t, \hat \alpha (t,\underline X _t, \underline \mu _t, \underline Y _t )) \\
  & \quad  \geq 
  D_x h (t, \underline X _t, \underline \mu _t^{n+1}, \hat \alpha (t, \underline X _t, \underline \mu _t^{n+1}, \underline Y _t ))- D_x h (t, \underline X _t, \underline \mu _t, \hat \alpha (t, \underline X _t, \underline \mu _t, \underline Y _t )) 
  \geq  0,
\end{align*}
so that 
\begin{align*}
 & | D_x h (t, \underline X _t, \underline \mu _t^n, \hat \alpha (t, \underline X _t, \underline \mu _t^n, \underline Y _t ))- D_x h (t, \underline X _t, \underline \mu _t, \hat \alpha (t, \underline X_t, \underline \mu _t, \underline Y _t )) |^2  \\ \notag 
  & \quad  \geq 
  | D_x h (t, \underline X _t, \underline \mu _t^{n+1}, \hat \alpha (t, \underline X _t, \underline \mu _t^{n+1}, \underline Y _t ))- D_x h (t, \underline X _t, \underline \mu _t, \hat \alpha (t, \underline X _t, \underline \mu _t, \underline Y _t )) |^2.\notag
\end{align*} 
Also, by the growth conditions on $D_x h$ (from Condition \ref{assumption existence: growth conditions on costs} in Assumption \ref{assumption existence}) and on $\hat \alpha$ (from Remarks \ref{remark Lipschitz feedback}), from \eqref{eq a priori of everything} we obtain
\begin{align} 
\label{eq bounds for Dx for convergence}
 \mathbb E \bigg[ \int_0^T & | D_x h (t, \underline X _t, \underline \mu _t^0, \hat \alpha (t, \underline X _t, \underline \mu _t^0, \underline Y _t )) - D_x h (t, \underline X _t, \underline \mu _t, \hat \alpha (t, \underline X_t, \underline \mu _t, \underline Y _t )) |^2  dt \bigg]  \\ \notag
&   \leq C \mathbb E \bigg[ \int_0^T \Big( 1 + |\underline \Gamma _T|^2 + |\overline \Gamma _T|^2  + |\underline Y _t |^2 \Big) dt \bigg] < \infty, \notag
\end{align}
where the latter estimates follows from Lemma \ref{lemma a priori estimates} and from \eqref{eq BSDE integrability}.
Moreover, in case of Assumption \ref{assumption for J} we have that $\hat \alpha$ does not depend on the measure $\nu$, while under Assumption \ref{assumption for comparison}, we have that the function $\hat \alpha$ is continuous in $\nu$ (see Remark \ref{remark continuity feedback in mu}).
Hence, in both cases we have 
\begin{equation}\label{eq limit for Dx for convergence}
\lim_n | D_x h (t, \underline X _t, \underline \mu _t^n, \hat \alpha (t, \underline X _t, \underline \mu _t^n, \underline Y _t ))- D_x h (t, \underline X _t, \underline \mu _t, \hat \alpha (t, \underline X _t, \underline \mu _t, \underline Y _t ))|^2 = 0.
\end{equation}
Therefore, by using \eqref{eq estimate integral 2}, from the limits in \eqref{eq limit for Dx for convergence} and the estimates in \eqref{eq bounds for Dx for convergence}, we can invoke the monotone convergence theorem in order to obtain 
\begin{equation}
    \label{eq convergence integral 2}
    \lim_n \E \bigg[ \int_0^T |\hat H_t ^{2,n} |^2 dt \bigg] =0.
\end{equation}

Finally, combining together \eqref{eq convergence integral 1} and \eqref{eq convergence integral 2} into \eqref{eq estimate stability integral}, we obtain
$$
\lim_n \E \bigg[ \int_0^T | \hat h (t, \underline X_t^n, \underline \mu _t^{n-1}, \underline Y _t)  - \hat h (t,\underline X _t, \underline \mu _t, \underline Y _t) |^2 dt \bigg ] =0.
$$ 
The latter can be plugged, together with \eqref{eq convergence Dg}, into the first estimate \eqref{eq estimate stability BSDE}, so to conclude that
\begin{equation}
    \label{eq convergence BSDE}
\lim_n \E \bigg[ \sup_{ t \in [0,T]} |\underline{Y}_t^n - \underline Y _t|^2 + \int_0^T ( | \underline{Z}_t^n - \underline Z _t |^2+ | \underline{Z} _t^{\circ,n} - \underline Z _t^\circ |^2 )dt \bigg]=0,
\end{equation}
which completes the proof of the convergence of $(\underline Y^n, \underline Z^n, \underline Z ^{\circ,n})$. 

\smallbreak\noindent 
\emph{Step 3.} 
In this step we will prove that the process $\underline X$ is the unique strong solution to the SDE \eqref{eq controlled SDE} controlled by the feedback $\hat \alpha ( t, \cdot, \underline \mu _t, \underline Y _t)$. 

Indeed, by Lipschitz continuity of the data $b_1, \sigma, \sigma^\circ$ and of the feedback $\hat \alpha$ (see Remark \ref{remark Lipschitz feedback}), there exists a unique strong solution $X$ to the SDE   
\begin{equation}
\label{eq limit SDE}
 d  X _t =  \hat b ( t,  X _t, \underline \mu _t, \underline Y _t)  dt + \sigma (t,  X_t) d W_t + \sigma^\circ (t,  X_t) d B_t,  \quad  X _0 = \xi. 
\end{equation} 
By \eqref{eq FBSDE n}, the process $\underline X ^n$ solves the SDE
\begin{equation*}
 d \underline X _t ^n =  \hat b ( t, \underline X _t^n, \underline \mu _t^{n-1}, \underline Y _t^n) dt + \sigma (t, \underline X_t^n) d W_t + \sigma^\circ (t, \underline X_t^n) d B_t, \quad  \underline X _0^n = \xi.  
\end{equation*} 
 Therefore, by standard stability results for SDEs, we obtain 
 \begin{align} 
     \label{eq stability SDE}
     \E \bigg[ \sup_{t \in [0,T]} | \underline X _t^n -  X _t |^2 \bigg] 
     \leq  C \Bigg( & \E \bigg[ \int_0^T   |b_1(t,  X_t, \underline \mu _t^{n-1}) - b_1(t,  X _t, \underline \mu _t)|^2 dt \bigg] \\ \notag
     & \quad + \E \bigg[ \int_0^T |\hat \alpha ( t,  X _t, \underline \mu _t^{n-1}, \underline Y _t^n) - \alpha ( t,  X _t, \underline \mu _t, \underline Y _t) |^2  dt  \bigg] \bigg). 
\end{align}
We proceed by estimating separately the two terms in the right-hand side of \eqref{eq stability SDE}.

Under Assumption \ref{assumption for J}, the drift $b_1$ does not depend on the measure $\nu$, so that $$b_1(t,  X_t, \underline \mu _t^{n-1}) = b_1(t,  X _t, \underline \mu _t).$$ 
On the other hand, under Assumption \ref{assumption for comparison}, one has $$b_1(t,  X_t, \underline \mu _t^{n-1}) - b_1(t,  X _t, \underline \mu _t) = b_0(t, \underline \mu _t^{n-1}) - b_0(t, \underline \mu _t),$$ with $b_0$ continuous and bounded (from Assumption \ref{assumption a priori estimates}). 
Hence, thanks tot the limits in \eqref{eq point conv of mu n to mu}, by the dominated convergence theorem and  we obtain
\begin{equation} 
    \label{eq convergence drift SDE}
    \lim_n  \E \bigg[ \int_0^T   |b_1(t,  X_t, \underline \mu _t^{n-1}) - b_1(t,  X _t, \underline \mu _t)|^2 dt \bigg]  =0.
\end{equation}

We next estimate the term with $\hat \alpha$. 
By Lipschitz coninuity of $\hat \alpha$ (see Remark \ref{remark Lipschitz feedback}) and by the convergence established in \eqref{eq convergence BSDE} we have
\begin{equation}
    \label{eq convergence alpha  1} 
    \lim_n  \E \bigg[ \int_0^T |\hat \alpha ( t,  X _t, \underline \mu _t^{n-1}, \underline Y _t^n) - \alpha ( t,  X _t, \underline \mu _t^{n-1}, \underline Y _t) |^2  dt  \bigg] \leq C \lim_n \E \bigg[ \int_0^T | \underline Y _t^n -  \underline Y _t |^2  dt  \bigg] =0.  
\end{equation}
Moreover, under Assumption \ref{assumption for J}, $\hat \alpha$ does not depend on the measure $\nu$, so that $$\hat \alpha ( t,  X _t, \underline \mu _t^{n-1}, \underline Y _t) = \hat \alpha ( t,  X _t, \underline \mu _t, \underline Y _t). $$ 
On the other hand, under Assumption \ref{assumption for comparison}, we have shown in the proof of Proposition \ref{proposition best reply map increasing} (see \eqref{eq comparative alpha}) that $\hat \alpha$ is monotone in the measure $\nu$, therefore
$$
 \hat \alpha  (t,   X_t, \underline \mu _t, \underline Y _t) - \hat \alpha (t,  X_t, \underline \mu _t^n, \underline Y _t) \geq \hat \alpha (t,   X_t, \underline \mu _t, \underline Y _t) - \hat \alpha (t,  X_t, \underline \mu _t^{n +1}, \underline Y _t) \geq 0, 
$$
and
$$
 |\hat \alpha  (t,   X_t, \underline \mu _t, \underline Y _t) - \hat \alpha (t,  X_t, \underline \mu _t^n, \underline Y _t) |^2 \geq | \hat \alpha (t,   X_t, \underline \mu _t, \underline Y _t) - \hat \alpha (t,  X_t, \underline \mu _t^{n +1}, \underline Y _t) |^2 \geq 0.
$$
Also, by the growth conditions of $\hat \alpha$ (see Remark \ref{remark Lipschitz feedback}), from \eqref{eq a priori of everything} we can estimate
\begin{align*}
\E \bigg[ \int_0^T |\hat \alpha (t,   X_t, \underline \mu _t, \underline Y _t) & - \hat \alpha (t,  X_t, \underline \mu _t^{0}, \underline Y _t)|^2 dt \bigg] \\  & \leq C \E \bigg[ \int_0^T ( 1 + | X _t|^2 + |\underline Y _t| ^2  + \mathbb  |\underline \Gamma _t |^2 + | \overline \Gamma _t |^2 ) dt \bigg] < \infty, 
\end{align*}
where the latter integrability follows from Lemma \ref{lemma a priori estimates},  \eqref{eq BSDE integrability} and from the fact that $X$ is the solution of the SDE \eqref{eq limit SDE}. 
Therefore, by continuity of $\hat \alpha$ in the measure (see Remark \ref{remark continuity feedback in mu}), we can invoke the monotone convergence theorem in order to obtain  
\begin{equation}
    \label{eq convergence alpha 2} 
    \lim_n  \E \bigg[ \int_0^T |\hat \alpha ( t,  X _t, \underline \mu _t^{n-1}, \underline Y _t) - \alpha ( t,  X _t, \underline \mu _t, \underline Y _t) |^2  dt  \bigg] =0. 
\end{equation} 
Finally, combining \eqref{eq convergence drift SDE}, \eqref{eq convergence alpha  1} and \eqref{eq convergence alpha 2} into \eqref{eq stability SDE}, we get 
$$
\lim_n \E \bigg[ \sup_{t \in [0,T]} | \underline X _t^n -  X _t |^2 \bigg]  = 0.
$$
By uniqueness of the limit and by \eqref{eq convergence X n to X sup in L2}, we conclude that $\underline X = X$, which in turn implies that $\underline X$ is the unique strong solution to the SDE \eqref{eq limit SDE}. Moreover, by the previous limits, we obtain a stronger convergence with respect to \eqref{eq convergence X n to X sup in L2}; that is, 
\begin{equation}
    \label{eq convergence SDE with sup}
    \lim_n \E \bigg[ \sup_{t \in [0,T]} | \underline X _t^n -  \underline X _t |^2 \bigg]  = 0.
\end{equation}

\smallbreak\noindent
\emph{Step 4.}
Combining together \eqref{eq definition mu n and mu inside the proof} form Step 1, \eqref{eq limi BSDE with X} from Step 2 and the fact that $\underline X$ is the solution to \eqref{eq limit SDE} as seen in Step 3, we have that the process 
$(\underline X, \underline Y, \underline Z, \underline Z ^\circ)$ is a solution to the MKV FBSDE \eqref{eq MKV FBSDE}.  
Moreover, from \eqref{eq convergence BSDE} and \eqref{eq convergence SDE with sup} we obtain 
$$
\E \bigg[ \sup_{ t \in [0,T]} (|\underline{X}_t^n - \underline X _t|^2 + |\underline{Y}_t^n - \underline Y _t|^2) 
    + \int_0^T ( | \underline{Z}_t^n - \underline Z _t |^2+ | \underline{Z} _t^{\circ,n} - \underline Z _t^\circ |^2 )dt \bigg] =0,
$$
proving the claimed convergence. 

We finally prove that $(\underline X, \underline Y, \underline Z, \underline Z ^\circ)$ is the minimal solution to the MKV FBSDE.
Let $( X,  Y,  Z,  Z ^\circ)$ be another solution to the MKV FBSDE, set $\mu := ( \mathbb P _{\text{\tiny{$ X _t|B$}}})_{t\in [0,T]}$.  
Since $( X,  Y,  Z,  Z ^\circ)$ is a solution, as in Remark \ref{remark characterization of MFGE} we have $X = R(X) $, so that, by the definition of $\underline \Gamma$ (see \eqref{eq definition Gamma upper lower bounds}), we have $\underline X ^0 = \underline \Gamma \leq R(X) =X$. 
By monotonicity of $R$ (see Proposition \ref{proposition best reply map increasing}) we obtain $\underline X ^1 = R( \underline X ^0) \leq R(X)=X$. 
Therefore, by iterating the map $R$ we deduce that $\underline X ^n \leq X$ for any $n \in \mathbb N$, and, taking limits as in \eqref{eq point conv of X n to X}, we conclude that $\underline X \leq X$.
Thus,  $(\underline X, \underline Y, \underline Z, \underline Z ^\circ)$ is the minimal solution to the MKV FBSDE \eqref{eq MKV FBSDE}, which completes the proof of the theorem.
 
\subsection{Proof of Theorem \ref{theorem main}}
    \label{subsection proof second main result}
We first prove Claim \ref{theorem.main.existence}.
Thanks to Remark \ref{remark characterization of MFGE}, from Claim \ref{theorem.main.existence.FBSDE} in Theorem \ref{theorem main FBSDE} we deduce that the set of MFG equilibria $\mathcal M$ is nonempty. 

We now show the lattice property of $\mathcal M$.
Let $F$ be the set of solutions to the MKV FBSDE \eqref{eq MKV FBSDE}. 
By  Claim \ref{theorem.main.existence.FBSDE} in Theorem \ref{theorem main FBSDE}, the set $F$ is a complete lattice. 
This means that, given two solutions $\Sigma = (X,Y,Z,Z^\circ)$ and $\bar \Sigma = (\bar X , \bar Y, \bar Z, \bar Z ^\circ)$ of \eqref{eq MKV FBSDE}, there exist two (unique) solutions 
$$
\Sigma \land^{\text{\tiny{$F$}}} \bar \Sigma = ( X ^\land ,  Y ^\land,  Z ^\land,  Z ^{\circ, \land}) \quad \text{and} \quad \Sigma \lor^{\text{\tiny{$F$}}} \bar \Sigma =  ( X ^\lor ,  Y ^\lor,  Z ^\lor,  Z ^{\circ, \lor})
$$ 
of the MKV FBSDE \eqref{eq MKV FBSDE} which verify the conditions:
\begin{enumerate}
    \item $ X^\land \leq X, \, \bar X $; 
    \item $\tilde X \leq X^\land $ for any other solution $(\tilde X, \tilde Y, \tilde Z, \tilde Z ^\circ)$ of \eqref{eq MKV FBSDE} with $\tilde X \leq X, \, \bar X$;
    \item $ X, \, \bar X \leq  X^\lor$;
    \item $X^\lor \leq \hat X$ for any other solution  $(\hat X, \hat Y, \hat Z, \hat Z ^\circ)$ of \eqref{eq MKV FBSDE} with $ X, \, \bar X \leq \hat X$.
\end{enumerate}
Next, recall that the projection map $p:F \to \mathcal M$ is a bijection with inverse $\Gamma$ (see Remark \ref{remark characterization of MFGE}). 
Given $\mu, \bar \mu \in \mathcal M$, we can define the flows $\mu \land^{\text{\tiny{$\mathcal M$}}} \bar \mu:= p( \Gamma(\mu) \land^{\text{\tiny{$F$}}} \Gamma(\bar \mu))$ and $\mu \lor^{\text{\tiny{$\mathcal M$}}} \bar \mu:= p( \Gamma(\mu) \lor^{\text{\tiny{$F$}}} \Gamma(\bar \mu))$. 
Since  $\Gamma(\mu) \land^{\text{\tiny{$F$}}} \Gamma(\bar \mu) \leq  \Gamma(\mu) , \Gamma(\bar \mu) $, by monotonicity of the projection map $p$ (see \eqref{eq projection monotone}) we have $\mu \land^{\text{\tiny{$\mathcal M$}}} \bar \mu \leq^{\text{\tiny{\rm st}}}  \mu, \bar \mu$. 
Moreover, if $\nu \in \mathcal M$ with $\nu \leq^{\text{\tiny{\rm st}}}  \mu, \bar \mu $, then, by Proposition \ref{proposition best reply map increasing}, we have $\Gamma(\nu) \leq \Gamma(\mu) , \Gamma(\bar \mu)$, so that $\Gamma(\nu) \leq \Gamma(\mu) \land ^{\text{\tiny{$F$}}}  \Gamma(\bar \mu)$ and $\nu= p \circ \Gamma(\nu) \leq^{\text{\tiny{\rm st}}}  p(\Gamma(\mu) \land^{\text{\tiny{$F$}}}  \Gamma(\bar \mu)) = \mu \land^{\text{\tiny{$\mathcal M$}}} \bar \mu$. 
This proves that $\mu \land^{\text{\tiny{$\mathcal M$}}} \bar \mu$ is the biggest element $\nu$ of $\mathcal M$ (w.r.t.\ the order relation  $\leq^{\text{\tiny{\rm st}}} $) such that $\nu \leq^{\text{\tiny{\rm st}}}  \mu, \bar \mu$. 
Analogously, one can see that the stochastic flow $\mu \lor^{\text{\tiny{$\mathcal M$}}} \bar \mu:= p( \Gamma(\mu) \lor^{\text{\tiny{$F$}}} \Gamma(\bar \mu))$ is the smallest element $\nu$ of $\mathcal M$  such that $\mu, \bar \mu \leq^{\text{\tiny{\rm st}}}  \nu$. 
Therefore, the set $\mathcal M$ has a lattice structure compatible with the binary relation $\leq^{\text{\tiny{\rm st}}}$.
In particular, this argument shows that there exist minimal and maximal elements of $\mathcal M$, which are respectively given by
$$
( \mathbb P _{\text{\tiny{$ \underline X _t|B$}}})_{t\in [0,T]} \quad \text{and} \quad ( \mathbb P _{\text{\tiny{$ \overline X _t|B$}}})_{t\in [0,T]},
$$
where $\underline X$ and $\overline X$ are the forward components of the minimal and the maximal solutions to the MKV FBSDE \eqref{eq MKV FBSDE} (see Claim \ref{theorem.main.existence.FBSDE} in Theorem \ref{theorem main FBSDE}).

The proof of the Claim \ref{theorem.main.general.convergence} in Theorem \ref{theorem main} hinges on the proof of Claim \ref{theorem.main.general.convergence.FBSDE} in Theorem \ref{theorem main FBSDE}.
It is sufficient to prove the Claim \ref{theorem.main.general.convergence.up}, as the proof of Claim \ref{theorem.main.generalconvergence.down} follows by similar arguments. 

Recalling \eqref{eq point conv of mu n to mu}, we directly have 
$$
\underline \mu ^n \leq^{\text{\tiny{\rm st}}} \underline \mu ^{n+1},  \quad  n \in \mathbb N.
$$
Moreover, by elementary properties of the 2-Wasserstein distance, we have
\begin{equation}
    \label{eq first convergence mu n}
    \ell_{\mathcal C ^d, 2}^2 ( \underline \mu ^n , \underline \mu )  \leq  \mathbb E \bigg[  \sup_{t \in [0,T] } | \underline X _t^n  - \underline X _t |^2 \Big|  B \bigg],  \quad  n \in \mathbb N.   
\end{equation}
By \eqref{eq X n increasing}, \eqref{eq point conv of X n to X}, and \eqref{eq convergence SDE with sup}, we know that  
$$
 \lim _n  \sup_{t \in [0,T] } | \underline X _t^n  - \underline X _t |^2  = 0 ,  \quad  \mathbb P \text{-a.s.}
$$
Moreover, by \eqref{eq X n increasing} we have 
$$
\sup_{t \in [0,T] } | \underline X _t^n  -  \underline X _t |^2 \leq \sup_{t \in [0,T] } | \underline X _t^1  - \underline X _t |^2, \quad \text{ for $n\geq 1$,}
$$
and, since $\underline X ^1$ and $\underline X$ are both solutions to SDEs, we also have
$$
\E \bigg[ \sup_{t \in [0,T] } | \underline X _t^1  - \underline X _t |^2  \bigg] \leq 2 \bigg( \E \bigg[\sup_{t \in [0,T] } | \underline X _t^1 |^2 \bigg]  + \E \bigg[\sup_{t \in [0,T] } |\underline X _t |^2  \bigg] \bigg) < \infty.
$$
Therefore, by the monotone convergence theorem for conditional expectations, we can take limits in \eqref{eq first convergence mu n} in order to deduce that
$$
\lim_n \ell_{\mathcal C ^d, 2} ( \underline \mu ^n , \underline \mu )  \leq \lim _n   \bigg( \mathbb E \bigg[  \sup_{t \in [0,T] } | \underline X _t^n  - \underline X _t |^2  \Big|  B  \bigg] \bigg)^{1/2} =0,  \quad  \mathbb P \text{-a.s.}
$$
Similarly, taking limits in  expectation in \eqref{eq first convergence mu n}, by the monotone convergence theorem we conclude that 
$$
\lim_n  \mathbb E [\ell_{\mathcal C ^d, 2}^2 ( \underline \mu ^n , \underline \mu ) ]  \leq \lim _n  \mathbb E \bigg[  \sup_{t \in [0,T] } | \underline X _t^n  - \underline X _t |^2   \bigg] =0, 
$$ 
which completes the proof of the theorem. 

\subsection{Sketch of the proof of Theorem \ref{theorem main fictitious play}}
    \label{subsection proof convergence fictitious play}
In light of the analysis of Subsections \ref{subsection proof first main result} and \ref{subsection proof second main result}, we limit our self to prove the monotonicity of the sequences $(\hat X^n)_n$ and $(\hat \mu ^n)_n$, as well as their convergence to limit points $\hat X \in H^2$ and $\hat \mu \in \mathcal M _{\text{\tiny{$B$}}}^2$ such that $(\mathbb P _{\text{\tiny{$ \hat X _t|B$}}})_{t\in [0,T]} = \hat \mu$.   

We prove this monotonicity by induction.
Clearly, since $\hat X ^1 = \underline{\Gamma}$, by the definition of $\underline \Gamma$ (see \eqref{eq definition Gamma upper lower bounds}) we obtain $ \hat X^1 \leq \hat X ^2$.
Therefore, we have $\hat \mu ^1 =(\mathbb P _{\text{\tiny{$ \hat X _t^1|B$}}})_{t\in [0,T]} \leq ^{\text{\tiny{\rm st}}}  (\mathbb P _{\text{\tiny{$ \hat X _t^2|B$}}})_{t\in [0,T]}$, so to obtain
$$
  \hat \mu ^1 \leq^{\text{\tiny{\rm st}}} \frac{1}{2} \Big( (\mathbb P _{\text{\tiny{$ \hat X _t^2|B$}}})_{t\in [0,T]} + \hat \mu ^1 \Big) = \hat \mu ^2.
$$
Assume now that $\hat X ^{n-1} \leq \hat X ^n$ and $ \hat \mu ^{n-1} \leq^{\text{\tiny{\rm st}}} \hat \mu ^n $ for some $n \in \mathbb N _0$. 
By the definition of $\hat X ^{n+1}$ (see \eqref{eq fictitious play}) and  the monotonicity of the best reply map $\Gamma$ (see Lemma \ref{proposition best reply map increasing}), we have
$$
 \hat X^n = \Gamma (\hat \mu ^{n-1} )  \leq \Gamma (\hat \mu ^n ) = \hat X^{n+1}.
$$
Therefore, using again \eqref{eq fictitious play} and the monotonicity of the best reply map $\mathcal R$, we deduce that
\begin{align*}
\hat \mu ^n
& =  \frac{1}{n+1} (n \hat \mu ^{n} + \hat \mu^n )  \\
& \leq^{\text{\tiny{\rm st}}} \frac{1}{n+1} (\mathcal R (\hat \mu ^{n-1}) + (n-1) \hat \mu ^{n-1} + \hat \mu^n ) \\
& = \frac{1}{n+1} (\mathcal R (\hat \mu ^n) + (n-1) \hat \mu ^n + \hat \mu^n ) = \hat \mu ^{n+1}.
\end{align*}
Thus, the  sequences $(\hat X^n)_n$ and $(\hat \mu ^n)_n$ are increasing.

By monotonicity, one can define the limit processes
\begin{equation*}
\hat X  := \sup_n \hat X ^n= \lim_n \hat X ^n
 \quad \text{and} \quad 
 \hat \mu  := \sup_n \hat \mu ^n= \lim_n \hat \mu ^n, \quad \mathbb P \otimes \pi \text{-a.e.\ in $\Omega \times [0,T]$}. 
\end{equation*}
In particular, we have $\mathbb P _{\text{\tiny{$ \hat X _t^n|B$}}} \to \mathbb P _{\text{\tiny{$ \hat X _t|B$}}}$ weakly as $n \to \infty$, so that 
$$
\hat \mu _t^n : = \frac{1}{^n} \sum_{k=1}^n \mathbb P _{\text{\tiny{$ \hat X _t^k|B$}}} \to \mathbb P _{\text{\tiny{$ \hat X _t|B$}}}, \quad \text{weakly as $n \to \infty$}.
$$
Hence, by uniqueness of the limit, we have $( \mathbb P _{\text{\tiny{$ \hat X _t|B$}}})_{t\in [0,T]} = \hat \mu$.

The rest of the proof (i.e., the characterization of $\hat X$ and $\hat \mu$ as the minimal MFG solutions as well as the improved convergence in $\mathbb H ^2$ and in the 2-Wasserstein distances)
can now be recovered following the same lines as in the proof of Theorems \ref{theorem main FBSDE} and \ref{theorem main} (see Subsections \ref{subsection proof first main result} and \ref{subsection proof second main result}).

\section{Examples} 
    \label{section examples} 
In this section, we provide some examples meeting the assumptions made in Section \ref{section Problem formulation and main results}.

\subsection{Sufficient conditions for strong uniqueness for FBSDEs}
    \label{subsection Sufficient conditions for strong uniqueness for FBSDEs}
In the following, we provide two natural set-ups in which Assumption \ref{assumption uniqueness FBSDEs} is satisfied.
\subsubsection{Linear-convex models}
    \label{subsection linear convex models}
Let $b,\sigma, \sigma^\circ, h , g$ satisfy Assumptions \ref{assumption existence} and \ref{assumption a priori estimates}. 
Assume moreover that: 
\begin{align}
    \label{eq assumption linear convex}
    & b(t,x,\mu,a) = b_0 (t, \mu) + \bar b _1 (t) x + b_2(t) a  \\ \notag
    & \sigma^i ( t, x^i) = p^i (t) + q^i(t) x^i \\ \notag
    & \sigma^{i,\circ} ( t, x^i) = p^{i,\circ} (t) + q^{i,\circ} (t) x^i \\ \notag
    & \text{$h, g$ are convex in $(x,a)$.} \notag
\end{align}

Under the additional condition \eqref{eq assumption linear convex}, the existence of an optimal control (for any $ \mu \in \mathcal M _{\text{\tiny{$B$}}}^2$) can be shown as in Theorem 5.2 at p.\ 68 in \cite{Yong&Zhou99}, 
while uniqueness follows from the strict convexity of $h$ in $a$  (see Assumption \ref{assumption existence}).
In this case, the necessary and sufficient conditions of the stochastic maximum principle are fulfilled (see, e.g., \cite{CarmonaDelarue18}), so that any solution of the FBSDE \eqref{eq FBSDE optimal controls} provides an optimal control. 
Therefore, by uniqueness of the optimal control, the solution of FBSDE$_\mu$ \eqref{eq FBSDE optimal controls} is unique (on any stochastic basis) and Assumption \ref{assumption uniqueness FBSDEs} is satisfied.  

It is also worth underlining that, when Condition \eqref{eq assumption linear convex} holds, the continuity assumptions in the measure (see Assumption \ref{assumption existence}) are not needed for the existence of the optimal controls (see Lemma \ref{lemma existence of optimal controls}).

\subsubsection{Nondegenerate case} 
    \label{example assumption nondegenerate}
A second relevant example in which Assumption \ref{assumption uniqueness FBSDEs} is satisfied is when $A$ is compact and the effect of the noises is nondegenerate. 
For example, for $b,\sigma, \sigma^\circ, h , g$ which satisfy Assumption \ref{assumption existence}, we can enforce the additional requirements: 
\begin{enumerate} 
\item $A$ is compact;
 \item $d=d_1$ and the matrix $\sigma (t,x)$ is invertible with inverse $\sigma(t,x)^{-1}$ such that $\sigma(t,x)^{-1} \leq \bar K$ for any $(t,x) \in [0,T] \times \R^d $ and some $ \bar K >0$.  
\end{enumerate}
In this case, following the same rational as in the proof of Lemma 4.1 in \cite{nam2022coupled}, one can use a Girsanov transformation in order to prove that, for any $ \mu \in \mathcal M _{\text{\tiny{$B$}}}^2$, the FBSDE$_\mu$ \eqref{eq FBSDE optimal controls} admits a unique strong solution, so that Assumption \ref{assumption uniqueness FBSDEs} is fulfilled.
Clearly, the same conclusion can be obtained by requiring the same nondegeneracy condition on the matrix $(\sigma,\sigma^\circ) : [0,T] \times \R ^d \to \R ^{d \times d_1} \times \R ^{d \times d_2}$.

\subsection{Checking the submodularity conditions}
    \label{subsection Checking the submodularity conditions}
In the following, we illustrate how to check the submodularity conditions and we provide some examples.  

For $m \in \mathbb{N}$, consider a measurable function $\gamma : \R^d \times \R^m \to \R$. 
The function $\gamma$ is said to have \emph{decreasing differences} in $x$ and $y$ if 
$$
\gamma(\bar x , \bar  y) - \gamma (x , \bar  y) \leq \gamma(\bar x,y) - \gamma( x, y), \text{ for any $x, \bar x \in \R^d$, $\bar y, y \in \R^m$ with $x \leq \bar x$ and $y \leq \bar y$.}
$$
Observe that, if $\gamma \in C^2(\R^d \times \R^m)$, then $\gamma$ has decreasing differences in $x$ and $y$ if and only if 
\begin{equation}
    \label{eq decreasing differnces differential}
\frac{\partial^2 \gamma}{ \partial x^i \partial y^j}(x,y) \leq 0,  \quad \text{for each} \quad (x,y) \in \R^d \times \R^m, \ i=1,...,d, \ j=1,...,m. 
\end{equation}
Moreover, $\gamma$ is said to be \emph{submodular} in $x$ if 
$$
\gamma(x\land \bar x , y) + \gamma (x \lor \bar x, y) \leq \gamma(x,y) + \gamma(\bar x, y), \text{ for any $x, \bar x \in \R^d$, $y \in \R^m$.}
$$
This condition is always satisfied in the case $d=1$. 
If $d\geq 2$ and $\gamma(\cdot, y)$ is twice-differentiable in $x$, then $\gamma$ is submodular in $x$ if and only if, for each fixed $y$, for any $i=1,...,d$ the function $\gamma(\cdot, y)$ has decreasing differences in $x_i$ and $(x_j)_{j\ne i}$ (see Theorem 2.6.1 and Corollary 2.6.1 at p.\ 44 in \cite{Topkis11}).  
Hence, in the case of twice-differentiable cost functions $\gamma$, the submodularity in $x$ corresponds to having
\begin{equation}
    \label{eq submodular differential}
\frac{\partial^2 \gamma}{\partial x^i \partial x^j}(x,y)  \leq 0,  \quad \text{for each} \quad (x,y) \in \R^d \times \R^m, \ i=1,...,d  \text{ and }  j=1,...,d \text{ with } j \ne i.  
\end{equation}

For $n,m \in \mathbb N _0$ and a measurable function $\varphi=(\varphi^1,..., \varphi^m) : \R^n \to \R^m$, 
$\varphi$ is said to be nondecreasing (resp.\ nonincreasing) if $\varphi(x) \leq \varphi (\bar x)$ for every $ x, \bar x \in \R^n$ with $x \leq \bar x$ (resp.\ $\bar x \leq x$). 
Define the set
\begin{equation}
    \label{eq set increasing functions}
    \Phi_2^{n;m} :=  \{ \text{measurable nondecreasing $\varphi : \R^n \to \R^m$ with $\phi(x) \leq C (1+|x|^2)$, $C>0$} \}, 
\end{equation} 
and, for $\varphi \in \Phi_2^{d;m}$ and $\mu \in \mathcal P _2 (\R^d)$, set
$$
\langle \varphi, \mu \rangle := ( \langle \varphi^1, \mu \rangle, ..., \langle \varphi^m, \mu \rangle) ^{\text{\tiny {$\top$}}} := \big( \begin{matrix}
   \int_{\R^d} \varphi^1(x) d \mu(x), ... , \int_{\R^d} \varphi^m(x) d \mu(x)  
\end{matrix} \big) ^{\text{\tiny {$\top$}}}.
$$
It is straightforward to observe that, for two $\R^d$-valued square integrable r.v.'s $\zeta$ and $\bar \zeta$, if $\zeta \leq \bar \zeta$ $\mathbb{P}$-a.s., then $\langle \varphi, \mathbb P _\zeta \rangle \leq \langle \varphi, \mathbb P _{\bar \zeta} \rangle$ for any $\varphi  \in \Phi_2^{d;m}$.

All the considerations above easily allows to construct functions $\phi : \R^d \times \mathcal P _2 (\R^d) \to \R$ which are submodular in $x$ and have decreasing differences in $x$ and $\mu$.
\begin{example}[Mean-field interaction of scalar type]
    \label{example mean field interaction of scalr type}
    Consider a mean-field interaction of scalar type; that is, $\phi(x,\mu) =  \gamma(x, \langle \varphi, \mu \rangle )$ for given measurable maps $\gamma:\mathbb{R}^d \times \R^m \rightarrow \mathbb{R}$ and  $\varphi:\mathbb{R}^d \rightarrow\mathbb{R}^m$, $m \in \mathbb N _0$.
    If   $\varphi \in \Phi_2^{d;m} $  and the function $\gamma$ has decreasing differences in $x$  and $y$ and it is submodular in $x$, then the function $\phi$ has decreasing differences in $x$  and $\mu$ and it is submodular in $x$.
\end{example}
\begin{example}[Mean-field interactions of order-1]
    \label{example mean field interactions of order-1}
        Another example is provided by the interactions of order-1, i.e.\ when $\phi$ is of the form
        $
        \phi(x,\mu)=\int_{\R^d} \gamma(x,y) d\mu(y). 
        $
        It is easy to check that, if the function $\gamma$ has decreasing differences in $x$  and $y$ and it is submodular in $x$, then the function $\phi$ has decreasing differences in $x$  and $\mu$ and it is submodular in $x$. 
\end{example} 

For the sake of illustration, we present two explicit examples which verify either Condition \ref{assumption submodularity multidimensional for J} in Assumption \ref{assumption for J} or Condition \ref{assumption submodularity multidimensional with comparison}
in Assumption \ref{assumption for comparison}. 
\begin{example}
For a  function $\varphi \in \Phi_2^{d;m}$, $m \in \mathbb N _0$,
Condition \ref{assumption submodularity multidimensional for J} in Assumption \ref{assumption for J} holds in the following two cases:
    \begin{enumerate}
        \item (Convex costs) $d=m$ and $\phi(x,\mu) =  \gamma (x - \langle \varphi , \mu \rangle)$ or $\phi(x,\mu)=\int_{\R^d} \gamma(x - \varphi(y)) d\mu(y)$, for a convex function $\gamma \in C^2(\R^d )$, when either $d=1$ or $\gamma$ verifies the condition $\partial \gamma /\partial x^i \partial x^j =0$ for $i\ne j$ (this holds, e.g., for $\gamma ( x ) = | x |^2 $); 
        \item (Multiplicative costs) $\phi(x,\mu) =  \gamma_1 ( x ) \gamma_2 (\langle \varphi , \mu \rangle )$, for a  nonincreasing (resp.\ nondecreasing) function $\gamma_1 \in C^2(\R^d)$ and a nondecreasing (resp.\ nonincreasing) function $ \gamma_2 \in C^1(\R^m)$, such that either $d=1$ or $\gamma_2 \geq 0$ and $\gamma_1$ verifies \eqref{eq submodular differential}. 
    \end{enumerate}
\end{example}
\begin{example}
Consider, for $t \in [0,T]$ and $m \in \mathbb N _0$,  functions $ \varphi_t \in \Phi_2^{d;m}$ and $ \psi_t \in \Phi_2^{d;k}$, as well as functions $\gamma_1 \in C^1([0,T] \times \R^d)$ and $\gamma_2 \in C([0,T] \times \R^m)$. 
Take costs $h$ and $g$ of the form
\begin{align*}  
        & h(t,x,\mu,a) = \gamma_1 (t,x) \gamma_2 (t, \langle \varphi_t, \mu \rangle ) +l (a,\mu),\\
        & g(x,\mu) =\gamma_1 (T,x)  \gamma_2 (T,\langle \varphi_T, \mu \rangle) .
\end{align*}
Assume that the function $D_x \gamma_1 (t, x ) \gamma_2 (t, y) $ is nonincreasing in $(x,y)$ for any $t \in[0,T]$.
Then, Condition \ref{assumption submodularity multidimensional with comparison}
in Assumption \ref{assumption for comparison} is satisfied in the following two cases:
\begin{enumerate} 
        \item (Convex costs) $l(a,\mu) =  \gamma_3 ( a - \langle \psi _t, \mu \rangle )$ or $l(a,\mu) =\int_{\R^d} \gamma_3(a - \psi_t(y)) d\mu(y)$, for a convex function $\gamma_3 \in C^2(\R^k)$, when either $k=1$ or $\gamma_3$ verifies  $\partial \gamma_3 / \partial a ^i \partial a^j$ for $i\ne j$ (this holds, e.g., for $\gamma_3 ( a ) = | a  |^2$);   
        \item (Multiplicative costs) $l(a,\mu) =  \gamma_3 ( a ) \gamma_4 (\langle \psi _t, \mu \rangle )$, for a nonincreasing (resp.\ nondecreasing) function $\gamma_3 \in C^2(\R^k)$ and a nondecreasing (resp.\ nonincreasing) function $ \gamma_4 : \R^m \to \R$, such that $k=1$ or $\gamma_4 \geq 0$ and $\gamma_3$ verifies \eqref{eq submodular differential} in $a$. 
    \end{enumerate}
\end{example}

\subsection{Linear-quadratic submodular MFGs}
Building on the discussion of Subsections \ref{subsection Sufficient conditions for strong uniqueness for FBSDEs} and \ref{subsection Checking the submodularity conditions}, we now provide some more explicit examples of linear-quadratic MFGs that can be treated with the approach presented in Section \ref{section Problem formulation and main results}.
 
Take $b,\sigma,\sigma^\circ $ as in \eqref{eq assumption linear convex}. 
Consider, for $t \in [0,T]$ and $m \in \mathbb N _0$, functions $ \varphi_t \in \Phi_2^{d;d}$ and $ \psi_t  \in \Phi_2^{d;k}$, symmetric nonnegative semi-definite matrices
$P_t$  and $R_t$ and a matrix $Q_t $. 
\begin{example} 
     If $b_0$ does not depend on $\mu$, take $k=d$ and
    \begin{align*} 
        & h(t,x,\mu,a) = x P_t x + (x- \langle \varphi_t, \mu \rangle ) Q_t (x- \langle \varphi _t, \mu \rangle ) + a R_t a, \\
        & g(x,\mu) = x  P_T x + (x- \langle \varphi _T, \mu \rangle )  Q _T (x- \langle \varphi_T, \mu \rangle ).
    \end{align*} 
    All the assumptions of Theorems \ref{theorem main FBSDE} and Theorem \ref{theorem main} (in particular, Assumption \ref{assumption for J}) are fulfilled if
    \begin{enumerate}
        \item $(P_t)_j^i  + (Q_t)_j^i \leq 0,$ for $i, j=1,...,d$ with $ i \ne j$;   
        \item  $Q_t$ is symmetric nonnegative semi-definite and $ (Q_t)_j^i  \geq 0,$ for $i, j=1,...,d$;
        \item  either $| \varphi_t (x)| \leq K$ for any $(t,x) \in [0,T] \times \R^d$ or $A$ is compact. 
    \end{enumerate}    
    \end{example}
    
    \begin{example} 
    If $b_0$ depends on $\mu$,  assume that $b_0, \bar b _1$ and $b_2$ satisfy Condition \ref{assumption monotonicity multidimensional with comparison} in Assumption \ref{assumption for comparison} and take
    \begin{align*} 
        & h(t,x,\mu,a) = \langle \varphi_t, \mu \rangle Q_t x + a R_t a + (a  - \langle \psi_t, \mu \rangle ) P_t (a- \langle \psi_t, \mu \rangle )  \\
        & g(x,\mu) = \langle \varphi_T, \mu \rangle Q_T x.  
    \end{align*} 
    All the assumptions of Theorems \ref{theorem main FBSDE} and Theorem \ref{theorem main} (in particular, Assumption \ref{assumption for comparison}) are fulfilled if 
    \begin{enumerate}
        \item $(R_t)_j^i  + (P_t)_j^i \leq 0,$ for $i, j=1,...,d$ with $ i \ne j$;   
        \item  $ (P_t)_j^i  \geq 0,$ for $i, j=1,...,d$;
        \item $(Q_t)_j^i \leq 0$, for $i, j=1,...,d$;
        \item  Either $| \varphi_t (x)| + | \psi_t (x)| \leq K$ for any $(t,x) \in [0,T] \times \R^d$ or $A$ is compact. 
    \end{enumerate} 
    In this case, we point out that the conditions $\kappa (|a|^2 -1) \leq h(t,x,\mu,a) $ and $-\kappa \leq g (x,\mu) $
    in Assumption \ref{assumption existence} are not satisfied, but Lemma \ref{lemma existence of optimal controls} can be recovered via the sufficient conditions of the stochastic maximum principle (see Subsection \ref{subsection linear convex models}).
\end{example} 
In both the previous examples, the additional continuity requirement of Assumption \ref{assumption continuity in the measure} is satisfied, e.g., if
$\varphi_t $ and $\psi_t$ are bounded and continuous.

\appendix

\section{Proof of Lemma \ref{lemma existence of optimal controls}}\label{appendix}

Recall that the process $\mu$ is fixed. The proof is divided in three steps.
\smallbreak\noindent
\emph{Step 1.}
We introduce a weak formulation of the control problem and we use the results of \cite{HaussmannLepeltier90} to have a first existence result for optimal controls. 

A tuple $\rho=(\Omega^\rho, \mathcal{F}^\rho, \mathbb{F}^\rho,  \mathbb{P}^\rho, \xi^\rho, W^\rho, B^\rho, \mu^\rho, \alpha^\rho)$ is said to be an admissible weak control if 
\begin{enumerate}
\item $(\Omega^\rho, \mathcal{F}^\rho, \mathbb{F}^\rho,  \mathbb{P}^\rho)$, $\mathbb{F}^\rho = (\mathcal{F} ^\rho_t)_{t \in [0,T]}$, is a filtered probability space satisfying the usual conditions;
\item $\xi^\rho: \Omega^\rho \to \mathbb R ^d$ is an $\mathcal{F}_0^\rho$-measurable square integrable random variable;
\item $( W^\rho, B^\rho):\Omega^\rho \times [0,T] \to \R^{d_1} \times \R^{d_2}$ is an $(\Omega^\rho, \mathcal{F}^\rho, \mathbb{F}^\rho,  \mathbb{P}^\rho)$-Brownian motion; 
\item $\mu^\rho :\Omega^\rho \times [0,T] \to \mathcal P _2 (\R^d)$ is an $\mathbb F ^\rho$-progressively measurable process;
\item  $ \mathbb{P}^\rho \circ (\xi^\rho,  W^\rho, B^\rho, \mu^\rho)^{-1} = { \mathbb{P}} \circ (\xi,  W,  B, {\mu})^{-1}$;
\item $\alpha^\rho: \Omega^\rho \times [0,T] \to A$ is square integrable $\mathbb F ^\rho$-progressively measurable process.
\end{enumerate}
Let $\mathcal A ^w$ denotes the set of admissible weak controls. 
Given $\rho \in \mathcal A ^w$, we define the cost functional 
\begin{equation*} 
 J^w(\rho):= \mathbb{E}^\rho \bigg[ \int_0^T  h(t,{X}_t^{\rho}, \mu_t^\rho, \alpha_t ^\rho) dt + g({X}_T^{\rho}, \mu_T^\rho) \bigg], 
\end{equation*}
where $\mathbb E ^\rho$ denotes the expectation under the probability measure $\mathbb P ^\rho$ and the process ${X}^{\rho}=({X}^{1,\rho},...,{X}^{d,\rho})$ denotes the solution to the controlled SDE
$$ 
d{X}_t^{\rho}= ( b_1(t,{X}_t^{\rho}, \mu_t^\rho) + b_2(t) \alpha_t^\rho )dt + \sigma(t,{X}_t^{\rho}) dW_t^\rho + \sigma^{\circ} (t,{X}_t^{\rho}) dB_t^\rho,  \quad t \in [0,T], \quad {X}_0^{\rho}=\xi^\rho.
$$
By Assumption \ref{assumption existence}, such a solution $X^\rho$ exists unique on the stochastic basis $(\Omega^\rho, \mathcal{F}^\rho, \mathbb{F}^\rho,  \mathbb{P}^\rho)$ and it is $\mathbb{F}^\rho$-adapted.
A weak control $\bar \rho \in \mathcal A ^w $ is said to be optimal if $J^w(\bar \rho ) \leq J^w (\rho)$ for any $\rho \in \mathcal A ^w$.

In light of Assumption \ref{assumption existence}, Corollary 4.8 in \cite{HaussmannLepeltier90} (slightly adapted to our setting, in order to deal with the extra stochastic term $\mu$) ensures the existence of a weak control 
\begin{equation}
    \label{eq weak optimal control}
    \bar \rho =(\bar \Omega,\bar{ \mathcal{F}}, \bar{ \mathbb{F}}, \bar{ \mathbb{P}}, \bar \xi, \bar W, \bar B, \bar \mu,\bar \alpha) \in \mathcal A ^w,
\end{equation}
which minimizes the cost functional $J^w$ over $\mathcal A ^w$. 
Clearly, for any $\alpha \in \mathcal A $ one have that $\rho^\alpha =({\Omega}, {\mathcal{F}}, {\mathbb{F}}, { \mathbb{P}}, \xi, W, B, {\mu}, \alpha ) \in \mathcal A ^w$ and that $J(\alpha, \mu) = J^w(\rho^\alpha)$, so that 
\begin{equation}
 \label{eq weak controls are better}
    J^w (\bar \rho)= \inf _{\rho \in \mathcal A ^w} J^w (\rho) \leq \inf _{\alpha \in \mathcal A} J(\alpha, \mu).
\end{equation}

\smallbreak\noindent
\emph{Step 2.} We now characterize the control $\bar{\alpha}$ using the necessary conditions of the stochastic maximum principle.   
In order to do so, we underline that we will work on the stochastic basis  $(\bar{\Omega}, \bar{\mathcal{F}}, \bar{\mathbb{F}}, \bar{ \mathbb{P}})$ with noises $(\bar \xi,\bar W,\bar B, \bar{\mu})$. 
In particular, the filtration $\bar{\mathbb{F}}$ may be larger than the filtration generated by the noises $(\bar \xi,\bar W,\bar B, \bar{\mu})$. 

To simplify the notation, set $\bar X := {X}^{\bar{\rho}}$. 
Define next the adjoint processes $(\bar{Y}, \bar Z, \bar Z ^\circ, \bar M)$ as the solution of the BSDE
$$
d\bar Y _t = D_x H(t,\bar X _t, \bar \mu _t, \bar Y _t , \bar{\alpha}_t) dt + \bar Z _t d \bar W _t + \bar Z_t^\circ d \bar B _t + d \bar M_t, \quad \bar Y _T = D_x g (\bar X _T, \bar \mu _T).
$$
Since the control $\bar \rho$ is optimal for $J^w$ in the class of weak controls $\mathcal A ^w$, clearly it is optimal also in the smaller class 
$
\mathcal A _{\bar \rho}
$ of controls 
$\rho=(\Omega^\rho, \mathcal{F}^\rho, \mathbb{F}^\rho,  \mathbb{P}^\rho, \xi^\rho, W^\rho, B^\rho, \mu^\rho, \alpha^\rho) \in \mathcal A ^w$ 
 such that
 $(\Omega^\rho, \mathcal{F}^\rho, \mathbb{F}^\rho,  \mathbb{P}^\rho, \xi^\rho, W^\rho, B^\rho, \mu^\rho) =  (\bar \Omega,\bar{ \mathcal{F}}, \bar{ \mathbb{F}}, \bar{ \mathbb{P}}, \bar \xi, \bar W, \bar B, \bar \mu) $.
In other words, the control $\bar \alpha $ is optimal on the stochastic basis $(\bar \Omega,\bar{ \mathcal{F}}, \bar{ \mathbb{F}}, \bar{ \mathbb{P}}, \bar \xi, \bar W, \bar B, \bar \mu)$ determined by $\bar \rho$.
Therefore, by the stochastic maximum principle (see Theorem 1.59 at p.\ 98 in the Vol. II of \cite{CarmonaDelarue18}) we have that
$$
\bar{\alpha}_t \in \argmin _{a \in A} H(t,\bar X _t, \bar \mu _t, \bar Y _t , a), \quad \bar{\mathbb P} \otimes dt\text{-a.e.},
$$
so that, by Remark \ref{remark Lipschitz feedback}, we know that
$$ 
\bar{\alpha}_t = \hat \alpha (t,\bar X _t, \bar \mu _t, \bar Y _t), \quad \bar{\mathbb P} \otimes dt\text{-a.e.}
$$
This implies that, for coefficients $\hat b, \hat h, \hat g$ as in \eqref{eq definition data FBSDE}, the process $(\bar X, \bar Y,  \bar Z, \bar Z ^\circ, \bar M)$ is a solution of the FBSDE
\begin{align}\label{eq FBSDE on the new space with M}
   & d \bar X _t =  \hat b (t,\bar X _t,\bar \mu _t, \bar Y _t)dt + \sigma (t, \bar X _t)d \bar W _t + \sigma^\circ (t,\bar X _t)d \bar B _t, \quad  \bar X _0=\bar \xi, \\ \notag
    & d\bar Y _t = - \hat h (t,\bar X _t,\bar \mu _t, \bar Y _t) dt + \bar Z _t d \bar W _t + \bar Z _t^\circ d \bar B _t + d \bar M _t, \quad \bar Y _T = \hat g (\bar X _T, \bar \mu _T). \notag 
\end{align}
The system \eqref{eq FBSDE on the new space with M} is a FBSDE (in a random environment) with coefficients $(\hat b, \sigma, \sigma^\circ, \hat h, \hat g)$ and noises $(\bar \xi ,\bar W,\bar B, \bar{\mu})$ on the stochastic basis  $(\bar{\Omega}, \bar{\mathcal{F}}, \bar{\mathbb{F}}, \bar{ \mathbb{P}})$ (we refer to Chapter 1 in the Vol.\ II of \cite{CarmonaDelarue18} for further details on FBSDEs in a random environment). 

\smallbreak\noindent
\emph{Step 3.} We finally identify the optimal control $\alpha^\mu$ by reconstructing a copy of the solution of the FBSDE \eqref{eq FBSDE on the new space with M} on the original probability space. 

From the definition of weak control, we have  $\bar{\mathbb P} \circ (\bar \xi,\bar W ,\bar B , \bar{\mu})^{-1} =\mathbb P \circ (\xi, W, B, \mu)^{-1}$, so that the process $\bar{\mu}$ is adapted to the filtration generated by $\bar B$.  
This allows to use Remark 1.16 at p.\ 15 in Vol.\ II in \cite{CarmonaDelarue18} to deduce that the martingale term $\bar M$ is null, so that the process $(\bar X, \bar Y,  \bar Z, \bar Z ^\circ)$ is a solution to the FBSDE 
\begin{align}\label{eq FBSDE on the new space no M} 
   & d \bar X _t = \hat b (t,\bar X _t,\bar \mu _t, \bar Y _t)dt + \sigma (t, \bar X _t)d \bar W _t + \sigma^\circ (t,\bar X _t)d \bar B _t, \quad  \bar X _0=\bar \xi, \\ \notag
    & d\bar Y _t = -\hat h (t,\bar X _t,\bar \mu _t,\bar Y _t) dt + \bar Z _t d \bar W _t + \bar Z _t^\circ d \bar B _t, \quad \bar Y _T = \hat g (\bar X _T, \bar \mu _T). \notag 
\end{align}
Moreover, by Assumption \ref{assumption uniqueness FBSDEs}, such a system satisfies the strong uniqueness property. 
Therefore, by Theorem 1.33 at p.\ 34 in Vol. II of \cite{CarmonaDelarue18}, there exists a solution $(X,Y,Z,Z^\circ)$ of the FBSDE with coefficients  $(\hat b, \sigma, \sigma^\circ, \hat h, \hat g)$, and with noises $(\xi, W, B, {\mu})$ on the stochastic basis  $({\Omega}, {\mathcal{F}}, {\mathbb{F}}, { \mathbb{P}})$.
Moreover, such a solution is such that 
$$
\mathbb P \circ (X,Y,Z,Z^\circ, \xi, W, B, {\mu} )^{-1} = \bar{\mathbb P} \circ (\bar X, \bar Y,  \bar Z, \bar Z ^\circ, \bar \xi,\bar W ,\bar B , \bar{\mu})^{-1}.
$$
The latter, allows to conclude that 
$$
{\mathbb P} \circ ( X,  {\mu}, (\hat \alpha (t, X _t,  \mu _t,  Y _t))_{t \in [0,T]} )^{-1} = \bar{\mathbb P} \circ (\bar X,  \bar{\mu}, (\hat \alpha (t,\bar X _t, \bar \mu _t, \bar Y _t))_{t \in [0,T]} )^{-1}, 
$$
so that, setting $\alpha^\mu:= (\hat \alpha (t, X _t,  \mu _t,  Y _t))_{t \in [0,T]}$, one has $J(\alpha^\mu, \mu) =  J^w (\bar{\rho})$ and, by \eqref{eq weak controls are better}, we conclude that the control $\alpha ^\mu$ is optimal for the functional $J(\cdot, \mu)$ on $\mathcal A$. 

Finally, if $\tilde \alpha \in \mathcal A$ is another optimal control for $J(\cdot,\mu)$, we can use  the same arguments as in the beginning of this step in order to show that 
$\tilde \alpha  =( \hat \alpha (t,\tilde X _t,  \mu _t, \tilde Y _t))_{t \in [0,T]} $ with 
$(\tilde X, \tilde Y,  \tilde Z, \tilde Z ^\circ)$ is a solution of the FBSDE$_\mu$ \eqref{eq FBSDE optimal controls}.
Again, by strong uniqueness (Assumption \ref{assumption uniqueness FBSDEs}), we have $(X,Y,Z,Z^\circ)= (\tilde X, \tilde Y,  \tilde Z, \tilde Z ^\circ)$ and 
$\tilde \alpha  =( \hat \alpha (t,\tilde X _t,  \mu _t, \tilde Y _t))_{t \in [0,T]} = (\hat \alpha (t, X _t,  \mu _t,  Y _t))_{t \in [0,T]} = \alpha^\mu $, proving the uniqueness of the optimal control.


\smallskip
\textbf{Acknowledgements.} 
The author is grateful to Fran{\c{c}}ois Delarue, Giorgio Ferrari, Markus Fischer, Max Nendel and Jianfeng Zhang for fruitful conversations.
Funded by the Deutsche Forschungsgemeinschaft (DFG, German Research Foundation) - Project-ID 317210226 - SFB 1283 

\bibliographystyle{siam} 
\bibliography{main.bib}

\end{document}